\documentclass[onefignum,onetabnum]{siamart190516}

\usepackage{lipsum}
\usepackage{amsfonts,amsmath,amsopn}
\usepackage{amssymb}
\usepackage{enumitem}
\usepackage{dsfont}
\usepackage{graphicx}
\usepackage{mathtools}
\usepackage{epstopdf}
\usepackage{algorithm}
\usepackage{algorithmic}
\ifpdf
  \DeclareGraphicsExtensions{.eps,.pdf,.png,.jpg}
\else
  \DeclareGraphicsExtensions{.eps}
\fi

\DeclareMathOperator*{\minimize}{\text{minimize}}

\newcommand{\R}{\mathbb{R}}
\newcommand{\N}{\mathbb{N}}
\newcommand{\Rn}{\mathbb{R}^n}

\newcommand{\Prob}{\mathbb{P}}

\newcommand{\dist}{\mathrm{dist}}

\newcommand{\sL}{{\sf L}}
\newcommand{\sG}{{\sf G}}
\newcommand{\cO}{\mathcal O}

\newcommand{\RR}{{\sf RR }}
\newcommand{\RRp}{{\sf RR}}
\newcommand{\IG}{{\sf IG }}

\definecolor{purp}{RGB}{152,24,147}
\definecolor{bluep}{RGB}{0,128,255}
\definecolor{redp}{RGB}{255,0,0}
\definecolor{orangep}{RGB}{255,128,0}

\newcommand{\revise}{}

\newcommand{\iprod}[2]{\left\langle #1, #2 \right\rangle}

\DeclareMathOperator*{\argmin}{argmin}

\newcommand{\be}{\begin{equation}}
	\newcommand{\ee}{\end{equation}}

\newsiamremark{remark}{Remark}
\newsiamremark{hypothesis}{Hypothesis}
\crefname{hypothesis}{Hypothesis}{Hypotheses}
\newsiamthm{assumption}{Assumption}
\newsiamthm{claim}{Claim}

\headers{Convergence of  Random Reshuffling  Under The KL Inequality}{X. Li, A. Milzarek, and J. Qiu}

\title{Convergence of  Random Reshuffling  Under \\ The Kurdyka-{\L}ojasiewicz Inequality}

\author{
Xiao Li\thanks{School of Data Science (SDS), Shenzhen Institute of Artificial Intelligence and Robotics for Society (AIRS), The Chinese University of Hong Kong, Shenzhen, China  (\texttt{lixiao@cuhk.edu.cn}).}
\and Andre Milzarek\thanks{School of Data Science (SDS), Shenzhen Research Institute of Big Data (SRIBD), The Chinese University of Hong Kong, Shenzhen, China (\texttt{andremilzarek@cuhk.edu.cn}).}
\and Junwen Qiu\thanks{School of Data Science (SDS), Shenzhen Institute of Artificial Intelligence and Robotics for Society (AIRS), Shenzhen Research Institute of Big Data (SRIBD), The Chinese University of Hong Kong, Shenzhen, China (\texttt{junwenqiu@link.cuhk.edu.cn}).}
}

\ifpdf
\hypersetup{
  pdftitle={Convergence of Random Reshuffling Under The KL Inequality},
  pdfauthor={X. Li, A. Milzarek, and J. Qiu}
}
\fi

\begin{document}

\maketitle

\begin{abstract}
We study the random reshuffling ({\sf RR}) method  for smooth nonconvex optimization problems with a finite-sum structure.  Though this method is widely utilized in practice, e.g., in the training of neural networks, its convergence behavior is only understood in  several limited settings. In this paper, under the well-known Kurdyka-{\L}ojasiewicz (KL) inequality,  we establish strong limit-point convergence results for \RR with appropriate diminishing step sizes, namely, the whole sequence of iterates generated by \RR is convergent and converges to a single stationary point in an almost sure sense. In addition, we derive the corresponding rate of convergence, depending on the KL exponent and suitably selected diminishing step sizes. When the KL exponent lies in $[0,\frac12]$, the convergence is at a rate of $\cO(t^{-1})$ with $t$ counting the number of iterations. When the KL exponent belongs to $(\frac12,1)$, our derived convergence rate is of the form $\cO(t^{-q})$ with $q\in (0,1)$ depending on the KL exponent. The standard KL inequality-based convergence analysis framework only applies to algorithms with a certain descent property.  We conduct a novel convergence analysis for the \emph{non-descent} \RR method with \emph{diminishing step sizes} based on the KL inequality, which generalizes the standard KL framework.  We summarize our main steps and core ideas in an informal analysis framework, which is of independent interest. As a direct application of this framework, we establish similar  strong limit-point convergence results for the reshuffled proximal point method.
\end{abstract}

\begin{keywords} random reshuffling, Kurdyka-{\L}ojasiewicz framework, strong limit-point convergence, convergence rates
\end{keywords}

\begin{AMS} 90C30, 90C06, 90C26, 90C15
\end{AMS}

\section{Introduction}
In this paper, we consider the finite-sum optimization problem
\begin{equation}
	\label{eq:problem} \minimize_{x\in \Rn}~f(x) = \frac{1}{N}\sum_{i=1}^{N}f(x,i), 
\end{equation}
where each component function $f(\cdot,i) : \Rn \to \R$ is continuously differentiable but not necessarily convex. Problem \eqref{eq:problem} appears frequently in various engineering fields such as machine learning and signal processing \revise{\cite{bottou2018optimization,chi2019nonconvex}}. 
\revise{This work aims at studying the \emph{random reshuffling} ({\sf RR}) method for solving problem \eqref{eq:problem}. In the next subsection, we first introduce \RR and its special case, the incremental gradient method, in detail and then discuss motivational background for \RRp}. 

\subsection{The Random Reshuffling Method}\label{sec:alg}
We now review the core procedures of \RRp. First, let us define the set of all possible permutations of $\{1,2,\ldots, N\}$  as 
\be\label{eq:permutations}
\Lambda:= \{\sigma: \sigma \text{ is a permutation of } \{1,2,\ldots, N\}\}. 
\ee
At each iteration $t$, a  permutation $\sigma^t$   is generated according to an i.i.d. uniform distribution over $\Lambda$. Then, \RR updates $x^{t-1}$ to $x^{t}$ through $N$ consecutive gradient descent-type steps by using the components $\{f(\cdot, \sigma^t_1), \ldots, f(\cdot, \sigma^t_N)\}$ sequentially, where  $\sigma^t_i$ represents the $i$-th element of $\sigma^t$.  In each  step, only one component $f(\cdot, \sigma^t_i)$ is selected for updating. To be more specific,  this method starts with $\tilde x^t_0 = x^{t-1}$ and then uses $f(\cdot, \sigma^t_i)$ to update $\tilde x^t_{i}$  as 
\be\label{eq:RS update}
\tilde x^t_{i} = \tilde x^t_{i-1} - \alpha_t \nabla f(\tilde x^t_{i-1}, \sigma^t_i) 
\ee
for $i=1, 2, \ldots, N$, resulting in $x^t = \tilde x^t_N$. The code of \RR is shown in \Cref{algo:gm-shuff}. 

\begin{algorithm}[t]
	\caption{{\RRp}: Random Reshuffling}
	\label{algo:gm-shuff}
	\begin{algorithmic}[1]
		\REQUIRE {Choose an initial point $x^0\in\Rn$;}
		\STATE {Set iteration count $t=1$;}
		\WHILE {stopping criterion not met}
		\STATE {Update the step size $\alpha_t$ according to a certain rule;}
		\STATE{ Generate a uniformly random permutation $\sigma^{t}$ of $\{1,\ldots, N\}$;}
		\STATE {Set $\tilde x_0^{t}=x^{t-1}$; }
		\FOR {$i = 1,2,\ldots, N$}
		\STATE {Update the inner iterate via $\tilde x_{i}^{t}=\tilde x_{i-1}^{t}-\alpha_{t}\nabla {f}(\tilde x_{i-1}^{t},\sigma^{t}_{i})$;}
		\ENDFOR
		\STATE {Set $x^t = \tilde x_{N}^{t}$;}
		\STATE {Update iteration count $t = t+1$;}
		\ENDWHILE
	\end{algorithmic}
\end{algorithm}

A special case of \RR is the so-called \emph{incremental gradient method} ({\sf IG}), which uses $\sigma^t =  \{1,2,\ldots, N\}$  for all $t\geq 0$. Thus, it updates  $x^{t-1}$ to $x^{t}$ through \eqref{eq:RS update} by utilizing the components $\{f(\cdot, 1), \ldots, f(\cdot, N) \}$ sequentially. Hence, all our subsequent conclusions for \RR naturally apply to the incremental gradient method. 

\textbf{\revise{Motivations.}} One classical algorithm for addressing problem \eqref{eq:problem} is the gradient descent method. 
However, in many modern large-scale applications, a fundamental challenge of problem \eqref{eq:problem} is that the number of components $N$ can be extremely large. Therefore,  using the full information of $f$ (say the full gradient of $f$) in each update  can be  computationally prohibitive.   This observation is precisely one of the main motivations of \RRp, which is tailored to modern large-scale optimization problems with finite-sum structure.  In contrast to the standard gradient descent method, \RR  performs  a gradient descent-type step in each update by using the gradient of only \emph{one} component function rather than \emph{all} the components of $f$.  In terms of empirical performance, it has been reported that \RR can outperform the gradient descent method for large-scale instances of problem \eqref{eq:problem}; see \cite{bertsekas2011} for more details.  

\revise{Another popular and ubiquitous algorithm for solving problem \eqref{eq:problem} is the stochastic gradient method (SGD) \cite{robbins1951stochastic}. Though SGD is widely studied in theory, what is commonly implemented in practice is actually \RR  \cite{bottou2012,shamir2016,Nagaraj2019,haochen2019}. Indeed, \RR is also known as ``SGD without replacement'' \cite{shamir2016,Nagaraj2019}.  The superiority of \RR over SGD mainly stems from the following observations: 1) The random reshuffling sampling scheme is easier and faster to implement than the i.i.d. random sampling required in SGD. 2) \RR utilizes all the data points in each epoch and hence, often has better theoretical and empirical performance.
As a result of this superiority, \RR has been included in well-known software packages such as PyTorch \cite{paszke2019pytorch} and TensorFlow \cite{tensorflow2015} as a fundamental solver and is used in a vast variety of engineering fields.}  Most notably, \RR is extensively applied in practice for training deep neural networks; see, e.g., \cite{bertsekas2011,gurbu2019,haochen2019,Nagaraj2019,nguyen2020unified}. These special instances of \RR and {\sf IG} are part of the family of backpropagation algorithms for the training of neural networks \cite{bertsekas2011}.

Although \RR enjoys vast applicability in practice, its convergence behavior is primarily studied under strong convexity; see, e.g., \cite{gurbu2019}. This restrictive setting is mainly motivated by the claim that \RR significantly outperforms SGD for strongly convex problems. In this work, we will establish much broader convergence results for \RR under the Kurdyka-{\L}ojasiewicz (KL) inequality. We summarize our main contributions below. 

\subsection{Main Contributions}

In order to compare different convergence results, we adopt the terminologies ``strong limit-point convergence'' and ``weak convergence'' used in \cite{absil2005}. Here, weak convergence means $\lim_{t \to \infty} \|\nabla f(x^t)\|= 0$, while strong limit-point convergence requires the whole sequence of iterates to converge to a single limit point and this limit point is a stationary point of \eqref{eq:problem}.\footnote{Note that strong limit-point convergence and weak convergence are also known as ``sequence convergence'' and ``subsequence convergence'', respectively.}

We study \RR (see \Cref{algo:gm-shuff}) for solving problem \eqref{eq:problem}. Our ultimate goal is to understand the convergence behavior of this algorithm under the so-called Kurdyka-{\L}ojasiewicz (KL) inequality of $f$ (see \Cref{Def:KL-property}). With the standard Lipschitz continuous gradient assumption (see \Cref{Assumption:1}) and the mild KL assumptions (see \Cref{Assumption:2}), we establish \emph{strong limit-point convergence} results for \RR equipped with proper \emph{diminishing step sizes} (see \Cref{thm:finite sum,thm:sto-rate}). That is,  the \emph{whole} sequence of iterates $\{x^t\}_{t\geq 1}$ generated by \Cref{algo:gm-shuff} converges to a stationary point $x^*$ of $f$ in an almost sure sense. Next, under the same setting, we establish the corresponding \emph{rate of convergence} (see \Cref{thm:convergence rate,thm:sto-rate}), which will depend on the KL exponent $\theta$ of $f$ (see \eqref{eq:Lojasiewicz ineq}) and on the selected step sizes $\{\alpha_t\}_{t \geq 1}$. With a suitable choice of the diminishing step sizes $\{\alpha_t\}_{t \geq 1}$, our rate results can be summarized as:
\begin{itemize}
\item[1.] If $\theta\in [0,\frac{1}{2}]$ at the target stationary point $x^*$, then $\{x^t\}_{t\geq 1}$ will converge to $x^*$ at a rate of $\mathcal O(t^{-1})$. Remarkably, this rate matches the best available rate for $\RR$ obtained in strongly convex optimization; see, e.g., \cite{gurbu2019}.  
\item[2.] If $\theta \in (\frac{1}{2},1)$ at $x^*$, then $\{x^t\}_{t\geq 1}$ will converge to $x^*$ at a rate of $\cO(t^{-q})$ with $q\in (0,1)$ depending on the KL exponent $\theta$. 
\end{itemize}
A more detailed discussion of the possible rates of convergence and their dependence on the KL exponent $\theta$ and step sizes $\{\alpha_t\}_{t\geq1}$ is given in \Cref{sec:rate}. As a standard byproduct of the KL inequality-based analysis, all our convergence results apply to the last iterate of \Cref{algo:gm-shuff}. Our results extend the convergence analysis of \RR from the strongly convex regime to a much more general nonconvex setting.  To the best of our knowledge, all the above results are new for \RRp. 

\RR is a \emph{non-descent} method.  It can be viewed as a gradient descent method with an additional error term that is proportional to the step size \cite{BerTsi00,gurbu2019} and requires \emph{diminishing step sizes} to ensure convergence. Therefore, the standard KL inequality-based analysis developed in \cite{absil2005,AttBolRedSou10,AttBol09,AttBolSva13,BolSabTeb14}---which applies to algorithms that possess a certain sufficient decrease property---is not directly applicable here. Different to the standard KL framework, through invoking proper diminishing step sizes, we utilize a more elementary analysis to first establish weak convergence results that allows to avoid the sufficient decrease property. Then, by accumulating the potential ascent (i.e., error terms) of each step of {\sf RR}, we derive a novel auxiliary descent-type condition for the iterates. Based on this auxiliary condition, we combine the standard KL analysis framework and the dynamics of the diminishing step sizes to establish strong limit-point convergence results and derive the corresponding rate of convergence depending on the KL exponent $\theta$. Thus, our results generalize the standard KL analysis framework to cover a class of non-descent methods with diminishing step sizes. We summarize our main steps and core ideas in an informal analysis framework in \Cref{sec:framework}, which is of independent interest. As a direct application of this new KL analysis framework, we show similar convergence results for another important shuffling-type algorithm---the reshuffled proximal point method (see \Cref{sec:prox point}). 

We believe that our results and the general framework presented in \Cref{sec:framework} can also serve as a blueprint for the KL inequality-based convergence analysis of other non-descent methods with diminishing step sizes.

\subsection{Connections to Previous Works} 
\hspace{0.01cm}

$\bullet$ \textbf{The random reshuffling method.} There have been numerous works studying \RR for solving problem \eqref{eq:problem}. In what follows, we will briefly review some representative results for these methods and provide the connections to our results.



As pointed out by various works, e.g., \cite{bertsekas2011,bottou2009,gurbu2019}, \RR is widely utilized in practice for tackling large-scale machine learning tasks such as training neural networks. The short note \cite{bottou2009} shows empirical superiority of \RR over SGD.  \revise{Theoretically, understanding the observed superiority of \RR over SGD remains an important problem}.
In \cite{recht2012}, the authors conjectured a noncommutative arithmetic-geometric mean inequality that leads to a lower iteration complexity of \RR than that of SGD for least-squares optimization. However, this conjecture was later proved to be not true in general \cite{lai2020}. 
Recently, under the assumptions that the objective function $f$ is strongly convex, each component function $f(\cdot,i)$ has a Lipschitz continuous Hessian, and the iterates generated by \RR are uniformly bounded, G{\"u}rb{\"u}zbalaban et al. formally established that the whole sequence of q-suffix averaged iterates  converges to the unique optimal solution at a rate of $\cO({1}/{t})$ with high probability \cite{gurbu2019}.  The authors also commented that this rate result of \RR outperforms the min-max lower bound of the rate achieved by SGD for strongly convex minimization (i.e., $\Omega({1}/{\sqrt{t}})$; cf. \cite{agarwal2012,nemirovskij1983}). Initiated by these first insights, various works start to understand the properties of {\sf RR}, but mainly focusing on iteration complexity results; see, e.g., \cite{haochen2019,Nagaraj2019,mishchenko2020,nguyen2020unified,huang2021improved}. For instance, under the assumptions that $f$ is strongly convex and $f$, the gradient of $f$, and the Hessian of $f$ are all Lipschitz continuous, the work \cite{haochen2019} shows an $\cO(1/T)$ iteration complexity result in expectation for the distances between the iterates and the unique optimal solution, where  $T$ is the preset total number of iterations.  Relaxations of the different and partly stringent Lipschitz continuity conditions used in \cite{haochen2019} have been further discussed and investigated in \cite{Nagaraj2019,mishchenko2020}.
Note that there is not a direct way to compare our results to existing iteration complexity results since our KL inequality-based analysis is fundamentally different and describes asymptotic convergence behavior.

A common assumption imposed by the above mentioned works is that the objective function has to be strongly convex, which is typically not satisfied in many important applications. By contrast, our KL inequality-based convergence analysis applies to much more general nonconvex functions.  Specifically, when the KL exponent equals to $\frac{1}{2}$ at the stationary point that \RR converges to, the corresponding convergence rate will be $\cO({1}/{t})$, which  matches the rate of the strongly convex setting \cite{gurbu2019}. However, it should be noted that a function with a stationary point having KL exponent $\frac12$ can generally be nonconvex as well. 

As mentioned, \RR also covers the well-known and popular \IG method as a special case. In \cite{zhi1994analysis,tseng1998incremental,solodov1998incremental,solodov1998}, weak convergence results for \IG (i.e., $\lim_{t\rightarrow\infty}\|\nabla f(x^t)\| = 0$) are provided under various step size schemes and different assumptions on the problem. A more recent work \cite{gurbu2015} shows that if $f$ is strongly convex and the iterates generated by \IG are uniformly bounded, then the whole sequence of iterates will converge to the unique optimal solution at a rate of $\mathcal O(1/t)$. Our KL inequality-based analysis naturally extends these results for \IG to a broader nonconvex setting.

$\bullet$ \textbf{The KL inequality and the related convergence analysis framework.} 
\revise{The  Kurdyka-{\L}ojasiewicz (KL) inequality plays a key role in our analysis. We refer to a series of important works for the rich  theory and history of this property; see, e.g.,  \cite{lojasiewicz1959,lojasiewicz1963,lojasiewicz1993,kurdyka1998,kurdyka1994,BolDanLew-MS-06,BolDanLew06,BolDanLewShi07,BolDanLeyMaz10}.} The KL inequality is a powerful concept applicable to a large range of (nonconvex) functions arising in real-world problems; see \cite[Section 5]{BolSabTeb14} for a discussion. A striking application of this inequality  is to analyze the convergence (or the stronger finite length property) of the gradient flow in dynamical systems theory.  
Absil et al. \cite{absil2005} extended the analysis to obtain strong limit-point convergence  of numerical optimization algorithms when used to optimize real analytical functions, where the algorithms are assumed to satisfy a certain sufficient decrease property.  The results in \cite{absil2005} apply to the well-known gradient descent and trust region methods with appropriate line-search procedures.  The KL inequality-based convergence analysis was then extended and popularized by Attouch and Bolte et al.  \cite{AttBol09,AttBolRedSou10,AttBolSva13,BolSabTeb14} for establishing the strong limit-point convergence of several descent methods for nonsmooth nonconvex  optimization. In particular, the work \cite{AttBol09} considers the proximal point method for minimizing functions satisfying the KL inequality and obtains strong limit-point convergence results and the corresponding rate of convergence based on the KL exponent.
The strong limit-point convergence result was then shown for  proximal coordinate-type (or alternating-type) gradient methods in \cite{BolSabTeb14} when optimizing the sum of a smooth nonconvex function and a nonsmooth one.  Thanks to the aforementioned pioneering works, the strength of the KL inequality has been widely recognized in both optimization and engineering societies.  Now, the KL inequality has become one of the main tools for understanding convergence properties of (mostly nonconvex) optimization algorithms.

The standard KL inequality-based convergence analysis crucially relies on the descent property of the algorithm. Let us take the analysis framework summarized in \cite[Section 3.2]{BolSabTeb14} as an example. As a first step, the sufficient decrease property typically allows to show weak convergence. Then, by combining the sufficient decrease property, the weak convergence result, and the KL inequality, strong limit-point convergence as well as the corresponding rate of convergence can be established for the descent method in question. Most existing works follow this analysis framework to analyze descent algorithms. Besides, there is also a set of works considering methods that do not have a descent property on the objective function itself; see, e.g., the Douglas-Rachford splitting-type methods \cite{li2016douglas}, inertial proximal gradient methods \cite{ochs2014,boct2016} and its alternating version \cite{pock2016},  alternating direction method of multipliers (ADMM) \cite{li2015global,wang2019,boct2020}, to name a few. An almost universal proof strategy in these works is to construct a particular Lyapunov function (or merit function) for which the algorithm possesses a sufficient decrease property.  Then, the standard KL framework can be utilized to derive strong limit-point convergence results. The price to pay here is to assume the KL inequality on the Lyapunov function rather than the original objective function.  
This compensation can be mild if the goal is only to obtain strong limit-point convergence of the sequence of iterates since the KL inequality holds for a vast class of functions. Nonetheless, determining the KL exponent of the Lyapunov function given the exponent of the original objective function can be non-trivial.  

\revise{There are also a few works discussing the KL inequality-based analysis for descent-type methods with relative updating errors \cite{garrigos2015descent,frankel2015splitting}. The proof techniques in \cite{garrigos2015descent,frankel2015splitting} share certain similarities with our derivation for establishing the finite length result by invoking a summability condition on the relative errors (cf. our \eqref{eq:festi-kl} and \cite[Page 896]{frankel2015splitting}). However, the algorithms considered by Garrigos and Frankel et al. still satisfy a sufficient decrease property and utilize step sizes that are bounded away from zero. Hence, the analysis conducted in \cite{garrigos2015descent,frankel2015splitting}  for obtaining weak convergence, strong limit-point convergence, and convergence rates follows the standard one \cite{AttBol09,BolSabTeb14} and is different from our derivations.}

To summarize, the existing works mainly consider algorithms that possess a sufficient decrease property either on the objective function itself or on a well constructed Lyapunov function. 
By contrast, \RR is a \emph{non-descent} method and requires \emph{diminishing step sizes} to ensure convergence. Therefore, the standard KL inequality-based convergence analysis framework pioneered by the works \cite{absil2005,AttBol09,AttBolRedSou10,AttBolSva13,BolSabTeb14} is not directly applicable here.  Instead, we conduct a novel analysis by combining the standard KL framework and the dynamics of the diminishing step sizes to establish strong limit-point convergence results and the corresponding rate of convergence for \RRp.
It is also worth emphasizing that our analysis only involves the KL inequality of the objective function itself.   We refer to \Cref{sec:framework} for a more detailed summary. 

\section{Preliminaries and Assumptions}
\subsection{The Kurdyka-{\L}ojasiewicz  (KL) Inequality}
We start with reviewing the basic elements of the KL inequality, which will serve as the key tool for our convergence analysis of \RRp.  As we consider the case where the objective function $f$ in problem \eqref{eq:problem} is smooth,  we now present the definition of the KL inequality tailored to smooth functions.  More broadly, using appropriate subdifferentials, the KL inequality can be introduced for nonsmooth functions as well; see, e.g., \cite{AttBolRedSou10,AttBol09,AttBolSva13,BolSabTeb14}. 

\begin{definition}[KL inequality]
	\label{Def:KL-property}
	The function $f$ is said to satisfy the KL inequality at a point $\bar x\in \Rn$ if there exist $\eta\in(0,\infty],$ a neighborhood $U$ of $\bar x$, and a continuous and concave function $\varrho: [0,\eta)\to\mathbb{R}_+$ with
	\[ \varrho \in C^1((0,\eta)), \quad \varrho(0) = 0, \quad \text{and} \quad \varrho^\prime(x) > 0 \quad \forall~x \in (0,\eta), \]
	such that for all $x \in U \cap \{x: 0 < |f(x) - f(\bar x)| < \eta\}$ the KL inequality holds, i.e.,
	\be \label{eq:kl-ineq} \varrho^\prime(|f(x)-f(\bar x)|) \cdot \|\nabla{f}(x)\|\geq1. \ee 
\end{definition}
The KL inequality in \cref{Def:KL-property} is a slightly stronger variant of the classical condition, see, e.g., \cite[Definition 3.1]{AttBolRedSou10} for comparison. In fact, the KL inequality \eqref{eq:kl-ineq} is usually stated without taking the absolute value of $f(x) - f(\bar x)$. However, the KL inequality shown in \Cref{Def:KL-property} also holds for semialgebraic and subanalytic functions, which underlines its generality; see  \cite{BolDanLew-MS-06,BolDanLew06,BolDanLewShi07}. We refer to the paragraph after the proof of \cref{thm:finite sum} for further discussion. 


Let $ \mathcal S_\eta$ denote the class of all continuous, concave functions $\varrho : [0,\eta) \to \R_+$ satisfying the requirements in \cref{Def:KL-property}. Functions in this class are typically referred to as desingularizing functions. If the desingularizing function $\varrho \in \mathcal S_\eta$ additionally satisfies the following quasi-additivity-type property
\be \label{eq:qrho} [\varrho^\prime(x+y)]^{-1} \leq C_\rho[(\varrho^\prime(x))^{-1}+(\varrho^\prime(y))^{-1}] \quad \forall~x,y \in (0,\eta) \;  \text{with} \; x+y < \eta, \ee
for some $C_\rho > 0$, then we will write $\varrho \in \mathcal Q_\eta$. Let us further define 
\be\label{eq:Lojasiewicz function}
\mathcal L := \{\varrho : \R_+ \to \R_+: \exists~c > 0, \, \theta \in [0,1): \varrho(x) = cx^{1-\theta}\}. 
\ee
The desingularizing functions in $\mathcal L$ are called {\L}ojasiewicz functions. This class of functions obviously satisfies the conditions in \cref{Def:KL-property}, i.e., we have $\mathcal L \subset \mathcal S_\eta$ for all $\eta > 0$. If $f$ satisfies \eqref{eq:kl-ineq} with $\varrho(x) = cx^{1-\theta}$, $c > 0$, and $\theta \in [0,1)$ at some $\bar x$, then we say that $f$ satisies the KL inequality at $\bar x$ with exponent $\theta$. We note that the class of {\L}ojasiewicz functions $\mathcal L$ is a subset of $\mathcal Q_\eta$ for any $\eta > 0$. To clarify this claim, let us consider $\varrho\in \mathcal L$, i.e., we have $\varrho(x) = cx^{1-\theta}$ with exponent $\theta \in [0,1)$. Then, using $\varrho^\prime(x) = c(1-\theta)x^{-\theta}$ and the subadditivity of the mapping $x \mapsto x^\theta$, it readily follows $[\varrho^\prime(x+y)]^{-1} \leq (\varrho^\prime(x))^{-1} + (\varrho^\prime(y))^{-1}$ for all $x,y$. Consequently, in this case the constant $C_\rho$ in \eqref{eq:qrho} can be set to $C_\rho = 1$. Note that the class $\mathcal Q_\eta$ is generally larger than the set of {\L}ojasiewicz functions $\mathcal L$. In particular, the mapping $\varrho(x) = \log(1+x)$ is an example of a function satisfying $\varrho \in \mathcal Q_\infty$ but $\varrho \notin \mathcal L$. In fact, it is easy to verify $\varrho \in \mathcal S_\infty$ and for all $x,y > 0$, we have $[\varrho^\prime(x+y)]^{-1} = 1 + x + y \leq (1+x) + (1+y) = (\varrho^\prime(x))^{-1}+(\varrho^\prime(y))^{-1}$.

\subsection{Assumptions} \label{sec:assum}
Throughout this paper, we impose the following Lipschitz gradient assumption on each component function $f(\cdot, i)$ in problem \eqref{eq:problem}. 
\begin{assumption} \label{Assumption:1} We consider the condition:
	\begin{enumerate}[label=\textup{\textrm{(A.\arabic*)}},topsep=0pt,itemsep=0ex,partopsep=0ex]
		\item \label{A1} For all $i \in \{1,\ldots, N\}$, $f(\cdot, i)$ is bounded from below and the gradient $\nabla f(\cdot,i)$ is Lipschitz continuous with parameter ${\sf L}$.
	\end{enumerate}
\end{assumption}
This assumption also implies that $\nabla f$ is Lipschitz continuous with constant $\sL$. Such type of condition is ubiquitous in the analysis of smooth optimization algorithms.

Suppose that the gradient of a continuously differentiable function $h:\R^n \rightarrow \R$ is Lipschitz continuous with parameter $\sL$. Then, applying the so-called descent lemma (see, e.g., \cite{nesterov2003}), it follows:
\be\label{eq:descent lemma}
h(y) \leq h(x) + \langle  \nabla h(x), y- x\rangle + \frac{{\sf L}}{2}\|x-y\|^2, \quad \forall \ x, y \in \R^n.
\ee 
Let us now set $h \equiv f(\cdot, i)$ for some $i\in \{1,2,\ldots, N\}$. Choosing $y = x - \frac{1}{\sL} \nabla f(x,i)$ in \eqref{eq:descent lemma}, we can infer:
\be\label{eq:grad upper bound}
\|\nabla f(x, i)\|^2\leq 2 {\sf L} (f(x, i)-  \bar f_{\min}), \quad i = 1,2, \ldots, N,
\ee
where $ \bar f_{\min} := \min_{1\leq i\leq N}\{ f_{\min}(\cdot, i)\}$ with $f_{\min}(\cdot, i)$ being a lower bound of the mapping $f(\cdot, i)$. This variance-type bound will play a central role in our convergence analysis. 

We notice that \cref{algo:gm-shuff} is stochastic in nature and that it formally generates a stochastic process of iterates $\{X^t\}_{t \geq 1}$. Here, randomness is caused by the choice of the random permutations $\{\sigma^t\}_{t\geq 1}$ in step 4 of the algorithm. In the following, we first formulate assumptions for a single trajectory $\{x^t\}_{t\geq 1}$ of the stochastic process $\{X^t\}_{t\geq 1}$ and discuss convergence properties of $\{x^t\}_{t\geq 1}$ conditioned on those assumptions. Since our analysis is solely based on deterministic techniques, this later allows to easily establish convergence results for $\{X^t\}_{t\geq 1}$ in an almost sure sense. A detailed summary and discussion of the stochastic convergence behavior of $\{X^t\}_{t\geq 1}$ is provided in \Cref{sec:sto-sure}.

Let $\{ x^t\}_{t\geq 1}$ be generated by \cref{algo:gm-shuff}. We define the associated set of accumulation points as
\be\label{eq:limit point set}
\mathcal C := \{ x \in \Rn: \exists~\text{a subsequence} \; \{t_\ell\}_\ell \; \text{such that} \;  x^{t_\ell} \to x \text{ as } \ell \to \infty\}. \ee
Notice that the set $\mathcal C$ is closed by definition. We now state our main KL assumptions.
\begin{assumption}
	\label{Assumption:2} 
	We assume the following conditions:
	\begin{enumerate}[label=\textup{\textrm{(B.\arabic*)}},topsep=0pt,itemsep=0ex,partopsep=0ex]
		\item \label{B1} The sequence $\{ x^t\}_{t\geq 1}$ generated by \Cref{algo:gm-shuff} is bounded.
		\item \label{B2} The function $f$ satisfies the following KL inequality on $\mathcal C$: For every $\bar x \in \mathcal C$ there are $\eta \in (0,\infty]$, a neighborhood $U$ of $\bar x$, and a function $\varrho \in \mathcal Q_\eta$ such that
		\[ 
	     	\varrho^\prime(|f(x) - f(\bar x)|) \cdot \|\nabla f(x)\| \geq 1 \;\; \forall~x \in U \cap \{x \in \Rn: 0 < |f(x) - f(\bar x)| < \eta\}. 
		\]
		\item \label{B3} The KL property of $f$ formulated in \ref{B2} holds for every $\bar x \in \mathcal C$ and each respective desingularizing function $\varrho$ can be chosen from the class of {\L}ojasiewicz functions $\mathcal L$. 
	\end{enumerate}
\end{assumption}

Some remarks about \Cref{Assumption:2} are in order. One main feature of the KL inequality is that it holds naturally for subanalytic or semialgebraic functions, and hence it holds for a very general class of problems arising in practice.     We refer to \cite[Section 5]{BolSabTeb14} for a related discussion.  Therefore, assumptions \ref{B2} and \ref{B3} are very mild.   
In contrast to the standard KL framework  \cite{AttBolSva13,BolSabTeb14}, we will mainly work with desingularizing functions from the classes $\mathcal Q_\eta$ in \ref{B2} (for deriving strong limit-point convergence) and $\mathcal L$ in \ref{B3} (for deriving the rate of convergence). As we will see in detail in the following discussions, this stronger requirement is necessary to cope with the missing descent of \RRp. \revise{Let us again emphasize that using desingularizing functions from either the class $\mathcal Q_\eta$ or even the smaller class $\cal L$ does not affect the generality of the KL inequality, as the class $\mathcal L$ is already admitted by the general semialgebraic and subanalytic functions; see, e.g., \cite[Theorem \L 1]{kurdyka1998} and \cite[Corollary 16]{BolDanLewShi07}.} Condition \ref{B1} is often imposed in the KL framework; see, e.g., \cite{AttBolRedSou10,AttBol09,BolSabTeb14}. In fact, we can derive the boundedness of the sequence of function values $\{f(x^t)\}_{t\geq 1}$ under \Cref{Assumption:1}; see \Cref{lemma:approx descent} for details. Therefore, \ref{B1} is automatically satisfied once the function $f$ has bounded level sets, which is mild.

\section{Convergence Analysis}\label{sec:convergence analysis}
Equipped with all the machineries, we now turn to the convergence analysis of \RRp.  
As mentioned in \Cref{sec:assum}, \RR formally generates a stochastic process of iterates $\{X^t\}_{t\geq 1}$.  In order to make our derivations more transparent, we will first derive the aforementioned convergence results for a single trajectory $\{x^t\}_{t\geq 1}$ that can be viewed as the iterates generated by {\sf IG} or \RR with any fixed order $\sigma^t$ in \Cref{algo:gm-shuff}. In \Cref{sec:sto-sure}, we will  state the convergence results for the stochastic process $\{X^t\}_{t\geq 1}$ in an almost sure sense. 

\subsection{Weak Convergence Results}\label{sec:subsequential}
In this subsection, we show several important results including a recursion for \RR and the weak convergence results,  which will serve as the basis of our subsequent  analysis. We first establish a useful estimate for the sequence of iterates $\{x^t\}_{t\geq 1}$ and function values $\{f(x^t)\}_{t\geq 1}$.

\begin{lemma}\label{lemma:descent}
	Suppose that  assumption \ref{A1} is satisfied.  Let $\{x^t\}_{t\geq 1}$ be generated by \cref{algo:gm-shuff} for solving problem \eqref{eq:problem} with step sizes $\alpha_t \in (0,{1}/{(\sqrt{2}{\sf L}N)}]$. Then, for all $t\geq 1$, we have
	\be \label{eq:esti-t0} \begin{aligned}
		f(x^t) - \bar f_{\min} & \leq (1+ 2\sL^3 N^3 \alpha_t^3) \left( f(x^{t-1})- \bar f_{\min} \right)  - \frac{N\alpha_t}{2}\|\nabla{f}(x^{t-1})\|^2 \\ & \hspace{4ex}  -\frac{1 - \sL N \alpha_t}{2N \alpha_t}\left\| x^t-x^{t-1} \right\|^2, 
	\end{aligned} \ee
	where $\bar f_{\min}$, defined below \eqref{eq:grad upper bound}, is a constant.
\end{lemma}
\begin{proof}
Using the descent lemma of $f$ (by letting $h\equiv f$ in \eqref{eq:descent lemma}) gives
%
\begingroup
\allowdisplaybreaks
	\begin{align}
		\nonumber f(x^t)  - \bar f_{\min} & \leq f(x^{t-1})  - \bar f_{\min} +\iprod{\nabla{f}(x^{t-1})}{x^t-x^{t-1}}+\frac{{\sf L}}{2}\|x^t-x^{t-1}\|^2\\
		\nonumber & \hspace{-8.5ex} = f(x^{t-1}) - \bar f_{\min} -N\alpha_t\iprod{\nabla{f}(x^{t-1})}{\frac{1}{N}{\sum}_{i=1}^{N}\nabla{f}(\tilde x_{i-1}^{t},\sigma^{t}_i) }+\frac{{\sf L}}{2}\|x^t-x^{t-1}\|^2\\
		\nonumber  & \hspace{-8.5ex} = f(x^{t-1}) - \bar f_{\min} -\frac{N\alpha_t}{2}\|\nabla{f}(x^{t-1})\|^2 -\frac{N\alpha_t}{2}\left\|\frac{1}{N}{\sum}_{i=1}^{N}\nabla{f}(\tilde x^{t}_{i-1},\sigma^{t}_i)\right\|^2\\
		& \hspace{-4.5ex} +\frac{N\alpha_t}{2}\left\|\frac{1}{N}{\sum}_{i=1}^{N}\nabla{f}(x^{t-1},\sigma^{t}_i)-\frac{1}{N}{\sum}_{i=1}^{N}\nabla{f}(\tilde x^{t}_{i-1},\sigma^{t}_i)\right\|^2 + \frac{{\sf L}}{2}\|x^t-x^{t-1}\|^2, \label{eq:recursion-1}
	\end{align}
\endgroup
where the \revise{last} equality is due to $\iprod{a}{b}=\frac{1}{2}(\|a\|^2+\|b\|^2-\|a-b\|^2).$  To further derive an upper bound for the right-hand side of \eqref{eq:recursion-1}, we can compute
\begingroup
\allowdisplaybreaks
	\begin{align*}
\left\|\frac{1}{N}{\sum}_{i=1}^{N}\nabla{f}(x^{t-1},\sigma^{t}_i)-\frac{1}{N}{\sum}_{i=1}^{N}\nabla{f}(\tilde x^{t}_{i-1},\sigma^{t}_i)\right\|^2 & \\ & \hspace{-40ex} \leq \frac{1}{N}{\sum}_{i=1}^{N} \|\nabla{f}(x^{t-1},\sigma^{t}_i)-\nabla{f}(\tilde x^{t}_{i-1},\sigma^{t}_i)\|^2  \leq\frac{{\sf L}^2}{N}{\sum}_{i=1}^{N}\|x^{t-1}-\tilde x^{t}_{i-1}\|^2,
	\end{align*}
\endgroup
	where we applied \ref{A1} in the last inequality. 
	Let  $V_t:= \sum_{i=1}^{N}\|x^{t-1}-\tilde x^{t}_{i-1}\|^2$.  By following the arguments in \cite[Lemma 5]{mishchenko2020}, we have
\begingroup
\allowdisplaybreaks
	\begin{align*}
		V_t & = {\sum}_{i=1}^{N}\|x^{t-1}-\tilde x^{t}_{i-1}\|^2 = \alpha_{t}^2{\sum}_{i=2}^{N}\left\|{\sum}_{j=1}^{i-1}\nabla f(\tilde{x}_{j-1}^t,\sigma_j^t)\right\|^2 \\
		& \leq 2 \alpha_{t}^2\sum_{i=2}^{N}\left[ \left\|{\sum}_{j=1}^{i-1} \nabla f(\tilde{x}_{j-1}^t,\sigma_j^t) - \nabla f(x^{t-1},\sigma_j^t) \right\|^2 + \left\|{\sum}_{j=1}^{i-1}\nabla f(x^{t-1},\sigma_j^t)\right\|^2 \right] \\
		& \leq 2\alpha_{t}^2 \sum_{i=2}^{N}(i-1)\left[ {\sf L}^2{\sum}_{j=1}^{i-1}\left\|\tilde{x}_{j-1}^t-x^{t-1}\right\|^2 + {\sum}_{j=1}^{i-1}\left\|\nabla f(x^{t-1},\sigma_j^t)\right\|^2 \right] \\
		& \leq 2\alpha_{t}^2 \sum_{i=2}^{N}(i-1)\left[ {\sf L}^2V_t + 2{\sf L} {\sum}_{j=1}^{i-1} \left( f(x^{t-1}, \sigma_j^t) - \bar f_{\min} \right) \right] \\
		&\leq  {\sf L}^2 N^2 \alpha_{t}^2 V_t + 2 \sL N^3 \alpha_t^2 (f(x^{t-1})- \bar f_{\min}),
	\end{align*}
\endgroup
	where we used \ref{A1} and \eqref{eq:grad upper bound} in the third and fourth lines, respectively. It follows
	\be\label{eq:v_t}
	\begin{aligned}
		V_t \leq \frac{2 \sL N^3 \alpha_t^2}{1- {\sf L}^2 N^2 \alpha_{t}^2 }  \left( f(x^{t-1})- \bar f_{\min} \right) 
		\leq 4 \sL N^3 \alpha_t^2  \left( f(x^{t-1})- \bar f_{\min} \right), 
	\end{aligned}
	\ee
	where the last inequality is due to $\alpha_t \in (0, {1}/({\sqrt{2} \sL N})]$.  Invoking the above estimates in \eqref{eq:recursion-1} yields
	\begin{align*}
		f(x^t) - \bar f_{\min} &\leq (1+ 2\sL^3 N^3 \alpha_t^3) \left( f(x^{t-1})- \bar f_{\min} \right)  - \frac{N\alpha_t}{2}\|\nabla{f}(x^{t-1})\|^2 \\
		&\hspace{6ex}  -\frac{N\alpha_t}{2}\left\|\frac{1}{N}{\sum}_{i=1}^{N}\nabla{f}(\tilde x^{t}_{i-1},\sigma^{t}_i)\right\|^2 + \frac{{\sf L}}{2}\|x^t-x^{t-1}\|^2.
	\end{align*}
	Finally, recognizing $\|\alpha_t \sum_{i=1}^{N}\nabla{f}(\tilde x^{t}_{i-1},\sigma^{t}_i)\| = \|x^t -  x^{t-1}\|$ gives the desired result. 
\end{proof}

Based on \Cref{lemma:descent}, we can further simplify the recursion on the sequence of iterates $\{x^t\}_{t\geq 1}$. 

\begin{lemma}
	\label{lemma:approx descent}
	Under the setting of \Cref{lemma:descent}, suppose further that the step sizes $\{\alpha_t\}_{t\geq 1}$ satisfy
	$\alpha_t \in (0,{1}/{(\sqrt{2}{\sf L}N)}]$ and $\sum_{t=1}^\infty \alpha_t^3 < \infty$.
	Then, for all $t\geq 1$, we have
	\be\label{eq:approximate descent}
	f(x^t) \leq  f(x^{t-1})   - \frac{N\alpha_t}{2}\|\nabla{f}(x^{t-1})\|^2   -\frac{1 - \sL N \alpha_t}{2N \alpha_t}\left\| x^t-x^{t-1} \right\|^2 + 2\sG \sL^3 N^3  \alpha_t^3.
	\ee
	Here, $\sG := ( f(x^{0})- \bar f_{\min}) \exp(\sum_{j=1}^{\infty}  2\sL^3 N^3 \alpha_j^3)$ is a finite positive constant.
\end{lemma}

\begin{proof}
First note that we have $1 - \sL N \alpha_t>0$ due to $\alpha_t \in (0,{1}/({\sqrt{2}{\sf L}N})]$. Thus, inequality \eqref{eq:esti-t0} in \cref{lemma:descent} implies
	\[
	f(x^t) - \bar f_{\min} \leq (1+ 2\sL^3 N^3 \alpha_t^3) \left( f(x^{t-1})- \bar f_{\min} \right)
	\]
	for all $t\geq 1$. Unrolling this recursion yields
	\[
	\begin{aligned}
		f(x^t) - \bar f_{\min} &\leq \left( f(x^{0})- \bar f_{\min} \right) {\prod}_{j=1}^{t} (1+ 2\sL^3 N^3 \alpha_j^3)\\
		&= \left( f(x^{0})- \bar f_{\min} \right) \exp\left( {\sum}_{j=1}^{t} \log\left(  1+ 2\sL^3 N^3 \alpha_j^3\right) \right)\\
		&\leq \left( f(x^{0})- \bar f_{\min} \right) \exp\left( {\sum}_{j=1}^{t}  2\sL^3 N^3 \alpha_j^3 \right), \quad \forall \ t\geq 1,
	\end{aligned}
	\]
	where the last inequality follows from the fact that $\log(1+a) \leq a$ whenever $a\geq 0$.
	Noticing $\sum_{j=1}^\infty \alpha_j^3 < \infty$, we obtain
	\be\label{eq:bound f}
	f(x^t) - \bar f_{\min} \leq \sG,  \quad \forall \ t\geq 1,
	\ee 
	where $\sG := ( f(x^{0})- \bar f_{\min}) \exp(\sum_{j=1}^{\infty}  2\sL^3 N^3 \alpha_j^3)$
	is a finite positive constant. Plugging \eqref{eq:bound f} into \eqref{eq:esti-t0} yields the desired result.
\end{proof}

The constant $\sG$ defined in \Cref{lemma:approx descent} can be large, if the value of the series $\sum_{j=1}^{\infty} \alpha_j^3$ is large. In \Cref{sec:step sizes}, we will provide a more explicit expression for $\sG$ by using some representative choices of the step sizes $\{\alpha_t\}_{t\geq 1}$.

Thanks to the recursion shown in \Cref{lemma:approx descent}, we can follow an elementary analysis (cf. \cite{BerTsi00,Ber16}) to verify weak convergence of \RRp. Our results in \cref{Prop:convergence} complement and extend the (mostly iteration complexity-type) results derived in \cite{mishchenko2020,nguyen2020unified}. 

\begin{proposition}
	\label{Prop:convergence}
	Suppose that condition \ref{A1} is satisfied.  Let $\{x^t\}_{t\geq 1}$ be generated by \cref{algo:gm-shuff} for solving problem \eqref{eq:problem} with step sizes $\{\alpha_t\}_{t\geq 1}$ fulfilling 
	\be\label{eq:ass-step}
	\alpha_t \in \left(0,\frac{1}{\sqrt{2}{\sf L}N}\right], \quad \sum_{t=1}^\infty \alpha_t = \infty, \quad \text{and} \quad \sum_{t=1}^\infty \alpha_t^3 < \infty. 
	\ee
	Then, $\{f(x^t)\}_{t\geq 1}$ converges to some $f^* \in \R$ and we have $\lim_{t\to\infty} \|\nabla{f}(x^t)\|=0$, i.e., every accumulation point of $\{x^t\}_{t\geq 1}$ is a stationary point of problem \eqref{eq:problem}.
\end{proposition}

\begin{proof}
	By \Cref{lemma:approx descent}, for all $t\geq 1$, we have 
	\be\label{eq:approximate descent repeat}
	f(x^t) \leq  f(x^{t-1})   - \frac{N\alpha_t}{2}\|\nabla{f}(x^{t-1})\|^2   -\frac{1 - \sL N \alpha_t}{2N \alpha_t}\left\| x^t-x^{t-1} \right\|^2 + 2\sG \sL^3 N^3  \alpha_t^3.
	\ee
	This shows that the sequence $\{f(x^t)\}_{t\geq 1}$ satisfies a supermartingale-type recursion. 
	\revise{Specifically, applying \Cref{lem:sup-conv-thm}} and using the lower boundedness of $f$ (as stated in assumption \ref{A1}) and the third condition in \eqref{eq:ass-step}, we can infer that $\{f(x^t)\}_{t\geq 1}$ converges to some $f^*\in \R$.
	
	Now,  telescoping the recursion \eqref{eq:approximate descent repeat}, we obtain
	\be \label{eq:esti-t2} 
	\frac{N}{2} \sum_{t=1}^\infty \alpha_t \|\nabla f( x^{t-1})\|^2 + \sum_{t=1}^\infty \frac{\|x^t-x^{t-1}\|^2}{8N\alpha_t} \leq f( x^0)  - f^* +  2\sG \sL^3 N^3 \sum_{t=1}^{\infty}\alpha_t^3  < \infty, 
	\ee
    where we applied $1-\sL N\alpha_t \geq \frac14$. The condition $\sum_{t=1}^\infty \alpha_t = \infty$ in \eqref{eq:ass-step} immediately implies that $\liminf_{t\to\infty}\|\nabla{f}(x^t)\|=0$.  In order to show that $\lim_{t\to\infty} \|\nabla{f}(x^t)\|=0$, let us assume on the contrary that $\{\|\nabla{f}(x^t)\|\}_{t\geq 1}$ does not converge to zero.  \revise{Then, there exists an $\varepsilon>0$ such that both of the conditions $	\|\nabla{f}({x}^{t})\|\geq 2\varepsilon$ and $\|\nabla{f}({x}^{t})\|< \varepsilon$ have to hold for infinitely many $t$. By mimicking the constructions used in \cite[Proposition 1]{BerTsi00} and \cite[Theorem 6.4.6]{conn2000trust},} we can extract two infinite subsequences $\{k_j\}_{j\geq 1}$ and $\{\ell_j\}_{j \geq 1}$ \revise{with $k_j<\ell_j < k_{j+1}$} such that 
	\be\label{eq:construct subsequence}
	\|\nabla{f}({x}^{k_j})\|\geq 2\varepsilon,\quad \|\nabla{f}({x}^{\ell_j})\|<\varepsilon,\quad\text{and}\quad \|\nabla{f}({x}^{t})\|\geq\varepsilon
	\ee
	%
 for all $t=k_j+1,\dots,\ell_j-1$. \revise{Note that we have assumed $\ell_j - k_j\geq2$ without loss of generality. Otherwise, we can just ignore the terms indexed by $t$ in \eqref{eq:construct subsequence}}. Combining this observation with \eqref{eq:esti-t2} yields
	\revise{	\[
		\infty>  {\sum}_{t=1}^{\infty}\alpha_{t+1}\|\nabla{f}(x^{t})\|^2 \geq  {\sum}_{j=1}^{\infty} {\sum}_{t=k_j}^{\ell_j-1}\alpha_{t+1} \| \nabla f(x^t)\|^2 \geq \varepsilon^2{\sum}_{j=1}^{\infty}{\sum}_{t=k_j}^{\ell_j-1}\alpha_{t+1},
		\]}
	which implies  
	\be\label{eq:contradict-1}
	\lim\limits_{j\rightarrow \infty} \ \beta_j := {\sum}_{t=k_j}^{\ell_j-1}\alpha_{t+1} = 0.
	\ee
Next, applying the triangle and Cauchy-Schwarz inequality, we obtain
\begin{align*}
	\|{x}^{\ell_j} - {x}^{k_j}\|  & \leq {\sum}_{t=k_j}^{\ell_j-1}\sqrt{\alpha_{t+1}} \left[\frac{\|{x}^{t+1}-x^{t}\|}{\sqrt{\alpha_{t+1}}}\right] \\ & \hspace{-4ex} \leq \left[ {\sum}_{t=k_j}^{\ell_j-1} \alpha_{t+1} \cdot {\sum}_{t=k_j}^{\ell_j-1} \frac{\|x^{t+1}-x^t\|^2}{\alpha_{t+1}} \right]^\frac12 \leq \sqrt{\beta_j} \cdot \left[ {\sum}_{t=1}^{\infty} \frac{\|x^{t}-x^{t-1}\|^2}{\alpha_{t}}\right]^\frac12. 
\end{align*}
On the one hand, upon taking the limit $j \rightarrow \infty$ in the above inequality, together with \eqref{eq:contradict-1} and \eqref{eq:esti-t2}, we have 
	\be \label{eq:contradict 1}
	\lim\limits_{j\rightarrow \infty} \ \|{x}^{\ell_j} - {x}^{k_j}\| = 0.
	\ee
On the other hand, combining \eqref{eq:construct subsequence}, the inverse triangle inequality, and the
	Lipschitz continuity of $\nabla{f}$, we have 
	\be \label{eq:contradict 2}
	\varepsilon\leq \left|\|\nabla{f}({x}^{\ell_j})\|-\|\nabla{f}({x}^{k_j})\| \right| \leq \|\nabla{f}({x}^{\ell_j})-\nabla{f}({x}^{k_j})\|\leq {\sf L}\|{x}^{\ell_j}-{x}^{k_j}\|.
	\ee
	We reach a contradiction to \eqref{eq:contradict 1} by taking $ j\rightarrow \infty$ in \eqref{eq:contradict 2}. Consequently, we  conclude that $\lim_{t\to\infty}\|\nabla{f}(x^t)\|=0$. The proof is complete.
\end{proof}
	
\subsection{Strong Limit-Point Convergence Under the Kurdyka-{\L}ojasiewicz  Inequality} \label{sec:sequential}
In this subsection, we establish one of the main results of this paper, i.e.,  the whole sequence $\{x^t\}_{t\geq1}$ generated by \Cref{algo:gm-shuff} for solving problem \eqref{eq:problem} converges to a stationary point of $f$. 

As direct consequences of \Cref{Prop:convergence}, we first collect several properties of the set of accumulation points $\mathcal C$ defined in \eqref{eq:limit point set}. 

\begin{lemma}\label{lemma:limit point set}
	Suppose that the conditions stated in \Cref{Prop:convergence} and assumption \ref{B1} are satisfied. Then, the following statements hold:
	\begin{enumerate}[label=(\alph*),topsep=0pt,itemsep=0ex,partopsep=0ex]
		\item  The set $\mathcal C$ is nonempty and compact. 
		\item  $\mathcal C \subseteq \mathrm{crit}(f) := \{x \in \Rn: \nabla f(x) = 0\}$. 
		\item  $f$ is finite and constant on $\mathcal C$. 
	\end{enumerate}
\end{lemma}

\begin{proof} Part (a) is an immediate consequence of \ref{B1}. The argument (b) follows directly  from  \Cref{Prop:convergence}. Finally, by \Cref{Prop:convergence}, we have $\lim_{t\rightarrow \infty} f(x^t) = f^*$ for some $f^*\in \R$. The continuity of $f$ then readily implies part (c). 
\end{proof}

Mimicking \cite[Lemma 6]{BolSabTeb14}, the observations in \Cref{lemma:limit point set} allow us to establish a uniformized version of the KL inequality for the class $\mathcal Q_\eta$ of quasi-additive desingularizing functions. 

\begin{lemma}\label{lemma:uniform KL}
	Suppose the assumptions formulated in \Cref{lemma:limit point set} and condition \ref{B2} (or \ref{B3}) are valid. Then, there are $\delta, \eta > 0$ and $\varrho \in \mathcal Q_\eta$ (or $\varrho \in \mathcal L$) such that for all $ \bar x \in \mathcal C$ and $x \in U_{\delta,\eta} := \{x \in \Rn: \mathrm{dist}(x,\mathcal C) < \delta\} \cap \{x \in \Rn: 0 < |f(x) - f(\bar x)| < \eta\}$, we have
	\be \label{eq:uni-kl} 
	      \varrho^\prime(|f(x) - f(\bar x)|) \cdot \|\nabla f(x)\| \geq 1. 
	\ee
\end{lemma} 

A detailed derivation of \cref{lemma:uniform KL} is presented in \cref{app:sec-unikl}. In the sequel, whenever we say that assumption \ref{B2} (or \ref{B3}) is satisfied, then we mean that it has to hold for the uniformized desingularizing function $\varrho$ in \eqref{eq:uni-kl}.

The standard convergence analysis based on the KL inequality applies to algorithms with a sufficient decrease property. However, the recursion shown in \Cref{lemma:approx descent} does not necessarily imply a descent property due to the error term $2\sG{\sf L}^3N^3\alpha_t^3$, rendering the standard KL analysis not directly applicable here.  
Fortunately, \Cref{lemma:approx descent} reveals an \emph{approximate descent property} of \Cref{algo:gm-shuff}. Based on this property, we derive a novel auxiliary descent-type condition (\eqref{eq:ut}-\eqref{eq:descent}) for the iterates which allows to express `descent' of the algorithm in an alternative way. We can then combine the standard KL analysis framework, e.g., \cite[Theorem 1]{BolSabTeb14}, and the dynamics of the diminishing step sizes to establish strong limit-point convergence results for \RRp.

We now present one of our main results in the following theorem. 
\begin{theorem}
	\label{thm:finite sum}
	Suppose that the assumptions \ref{A1}, \ref{B1}, and \ref{B2} are satisfied and let $\{x^t\}_{t \geq 1}$ be generated by \Cref{algo:gm-shuff}. In addition, let us assume  the following conditions on the step sizes  $\{\alpha_t\}_{t\geq 1}$: 
	\be\label{eq:kl-step}  
	\begin{split}
		{
			\alpha_t \in (0,{1} }&{/{(\sqrt{2}{\sf L}N)}],} \quad 
		{\sum}_{t=1}^\infty \alpha_t = \infty, \quad {\sum}_{t=1}^\infty \alpha_t^3 < \infty, \\ 
		& \text{and} \quad {\sum}_{t=1}^\infty \alpha_t \left[\varrho^\prime\left({\sum}_{j=t}^\infty \alpha_j^3\right)\right]^{-1} < \infty,
	\end{split}
	\ee
	where $\varrho \in \mathcal Q_\eta$ is the desingularizing function used in the uniformized KL inequality \eqref{eq:uni-kl}. Then, we have
	\begin{align*}
		{\sum}_{t=1}^{\infty}\alpha_t\|\nabla{f}(x^{t-1})\| <\infty  \quad\text{and} \quad {\sum}_{t=1}^{\infty}\|x^t-x^{t-1}\|<\infty.
	\end{align*}
	Consequently, the sequence $\{x^t\}_{t\geq 1}$ has finite length and converges to some stationary point $x^*$ of $f$. 
\end{theorem}

\begin{proof}
	For convenience, we repeat the recursion shown in \Cref{lemma:approx descent}. For all $t\geq 1$, we have 
	\be\label{eq:approximate descent repeat 2}
	f(x^t) \leq  f(x^{t-1})   - \frac{N\alpha_t}{2}\|\nabla{f}(x^{t-1})\|^2   -\frac{1 - \sL N \alpha_t}{2N \alpha_t}\left\| x^t-x^{t-1} \right\|^2 + 2\sG \sL^3 N^3  \alpha_t^3.
	\ee
	Let us define the accumulation of the error terms as
	\be\label{eq:ut}
	u_t = 2\sG \sL^3 N^3  {\sum}_{j=t+1}^{\infty}\alpha_{j}^3.
	\ee
	Upon plugging \eqref{eq:ut} into \eqref{eq:approximate descent repeat 2}, it holds that
	\begin{align}
		\label{eq:descent}
		f(x^t)+u_t\leq f(x^{t-1})+u_{t-1}-\frac{N\alpha_t}{2}\|\nabla{f}(x^{t-1})\|^2-\frac{1-{\sf L}N\alpha_t}{2N\alpha_t}\|x^t-x^{t-1}\|^2.
	\end{align}
	Using $\alpha_t \in (0,{1}/{(\sqrt{2}{\sf L}N)}]$, it holds that $1-{\sf L}N\alpha_t>0$ for all $t\geq 1$. Then, \eqref{eq:descent}  implies that the sequence $\{f(x^t)+u_t\}_{t\geq 1}$ is non-increasing.  Note that $f(x^t) \rightarrow f^* = f(\bar x)$  for all $\bar x \in \mathcal C$ as shown in \Cref{Prop:convergence} and \Cref{lemma:limit point set}.   Since $u_t\rightarrow 0$  by \eqref{eq:kl-step}, we have that the sequence $\{f(x^t)+u_t\}_{t\geq 1}$ is also convergent and converges to $f^*$. Let us set 
	\[
	\delta_t := \varrho(f(x^t) - f^* + u_t).
	\]
	Due to the monotonicity of the sequence $\{f(x^t)+u_t\}_{t\geq 1}$ and the fact that $f(x^t)+u_t \rightarrow f^*$,  $\delta_t$ is well defined as $f(x^{t}) - f^* + u_{t} \geq 0$ for all $t\geq 1$.  Then, for all $t$ sufficiently large  we obtain 
	\begin{align}
		\nonumber \delta_{t-1} - \delta_{t} & \geq \varrho^\prime(f(x^{t-1}) - f^* + u_{t-1}) \left [f(x^{t-1}) + u_{t-1} - f(x^{t}) - u_t \right] \\ 
		\nonumber & \geq \varrho^\prime(|f(x^{t-1}) - f^*| + u_{t-1}) \left [f(x^{t-1}) + u_{t-1} - f(x^{t}) - u_t \right]  \\
		\nonumber &\geq \varrho^\prime(|f(x^{t-1}) - f^*| + u_{t-1}) \left [  \frac{N\alpha_t}{2}\|\nabla{f}(x^{t-1})\|^2+\frac{1-{\sf L}N\alpha_t}{2N\alpha_t}\|x^t-x^{t-1}\|^2\right]\\
		&\geq \frac{1}{C_\rho}  \frac{\frac{N\alpha_t}{2}\|\nabla{f}(x^{t-1})\|^2+\frac{1-{\sf L}N\alpha_t}{2N\alpha_t}\|x^t-x^{t-1}\|^2}{[\varrho^\prime(|f(x^{t-1}) - f^*|)]^{-1} + [\varrho^\prime(u_{t-1})]^{-1}}, \label{eq:est-1}
	\end{align}
	where the first inequality is from the concavity of $\varrho$, the second inequality is due to the fact that $\varrho^\prime$ is monotonically decreasing (since $\varrho$ is concave),  the third inequality is from \eqref{eq:descent}, and the last inequality follows from \eqref{eq:qrho} (since $\varrho \in \mathcal Q_\eta$). 
	
Since $\{f(x^t)\}_{t\geq 1}$ converges to $f^*$ and we have $\dist(x^t,\mathcal C) \to 0$ by definition (cf.\ \eqref{eq:limit point set}), there exists $t_0 \in \N$ such that the uniformized KL inequality \eqref{eq:uni-kl} holds for all $x \equiv x^t \notin \mathcal C$ with $t \geq t_0-1$. In the following, without loss of generality, let us assume $x^t \notin \mathcal C$ or $u_t \neq 0$ for all $t \geq t_0-1$. Thus, applying the KL inequality \eqref{eq:uni-kl}  to \eqref{eq:est-1} yields 
\be \label{eq:est-3}
\begin{aligned} \delta_{t-1} - \delta_{t} \geq \frac{1}{C_\rho} \frac{\frac{N\alpha_t}{2}\|\nabla{f}(x^{t-1})\|^2+\frac{1-{\sf L}N\alpha_t}{2N\alpha_t}\|x^t-x^{t-1}\|^2}{\|\nabla f(x^{t-1})\| + [\varrho^\prime(u_{t-1})]^{-1}} \quad \forall~t \geq t_0. \end{aligned}
\ee
Then, for every $\vartheta \in [0,1)$ and all $t\geq t_0$, we have the following chain of inequalities:
\begingroup
\allowdisplaybreaks
\begin{align} 
	\nonumber \frac{\sqrt{N\alpha_t}}{2} \|\nabla f(x^{t-1})\| + \frac{\sqrt{1-{\sf L}N\alpha_t}}{2\sqrt{N\alpha_t}} \|x^t - x^{t-1}\| & \\
	\nonumber & \hspace{-40ex} \leq \left[ \frac{N\alpha_t}{2}\|\nabla{f}(x^{t-1})\|^2+\frac{1-{\sf L}N\alpha_t}{2N\alpha_t}\|x^t-x^{t-1}\|^2 \right]^{1/2} \\ 
	\nonumber & \hspace{-40ex} \leq \sqrt{C_\rho [\delta_{t-1}-\delta_t] [\|\nabla f(x^{t-1})\| + [\varrho^\prime(u_{t-1})]^{-1}]} \\ 
	& \hspace{-40ex} \leq \frac{C_\rho}{2\sqrt{N\alpha_t}(1-\vartheta)}[\delta_{t-1} - \delta_t] + \frac{\sqrt{N\alpha_t}(1-\vartheta)}{2} [\|\nabla f(x^{t-1})\| + [\varrho^\prime(u_{t-1})]^{-1}], \label{eq:est-4}
\end{align}
\endgroup
where the first inequality is from  the estimate $(a+b)^2 \leq 2(a^2 + b^2)$, the second inequality is due to \eqref{eq:est-3},  and the last inequality follows from the estimate $\sqrt{ab} \leq \frac{1}{2\varepsilon}a+\frac{\varepsilon}{2}b$ for any $a,b \geq 0$ and $\varepsilon > 0$. 

Multiplying both sides of \eqref{eq:est-4} with $2\sqrt{N\alpha_t}$ and rearranging the terms gives 
\be\label{eq:est-5} \begin{aligned}
N\alpha_t \vartheta \|\nabla f(x^{t-1})\| + \sqrt{1-{\sf L}N\alpha_t} \|x^t - x^{t-1}\| & \\ & \hspace{-25ex} \leq \frac{C_\rho}{1-\vartheta} [\delta_{t-1}-\delta_t] + N\alpha_t (1-\vartheta) [\varrho^\prime(u_{t-1})]^{-1}. \end{aligned}
\ee
for all $t \geq t_0$. Due to $\alpha_t \leq 1/(\sqrt{2}{\sf L} N) $, there further exists $\bar{\vartheta} > 0$ with $\sqrt{1-{\sf L}N\alpha_t} \geq \bar{\vartheta}$ for all $t \geq t_0$. Finally, summing the inequality \eqref{eq:est-5} over $t = t_0+1,\dots,T$ yields
\begingroup
\allowdisplaybreaks
\begin{align} \nonumber {\sum}_{t=t_0+1}^{T} \left[ N\vartheta \alpha_{t} \|\nabla f(x^{t-1})\| + \bar\vartheta \|x^{t}-x^{t-1}\|\right] & \\ \label{eq:festi-kl} & \hspace{-25ex}\leq \frac{C_\rho[\delta_{t_0}-\delta_T]}{1-\vartheta} + N(1-\vartheta) {\sum}_{t=t_0+1}^{T} \alpha_{t} [\varrho^\prime(u_{t-1})]^{-1}. \end{align} \endgroup
Since $\varrho$ is continuous with $\varrho(0) = 0$, we have $\delta_T \to 0$ as $T \to \infty$. Thus, taking the limit $T \to \infty$ in \eqref{eq:festi-kl} and invoking \eqref{eq:kl-step}, we obtain
\begin{align*}
	{\sum}_{t=1}^{\infty}\alpha_t\|\nabla{f}(x^{t-1})\| <\infty  \quad\text{and} \quad {\sum}_{t=1}^{\infty}\|x^t-x^{t-1}\|<\infty.
\end{align*}
By definition, the second estimate implies that $\{x^t\}_{t\geq 1}$ has finite length, and hence is convergent. This, together with the result that every accumulation point of  $\{x^t\}_{t\geq 1}$ is a stationary point of $f$ as shown in \Cref{Prop:convergence}, yields that $\{x^t\}_{t\geq 1}$ converges to some stationary point $x^*$ of $f$. 
\end{proof}

As can be observed from the statements of \Cref{lemma:descent}, \Cref{lemma:approx descent}, \Cref{Prop:convergence}, and \Cref{thm:finite sum}, we impose an increasing number of conditions on the step sizes $\{\alpha_{t}\}_{t\geq 1}$ in order to obtain stronger convergence results. Though the constant $\sG$ in \cref{lemma:approx descent} and the last condition in \eqref{eq:kl-step} are less intuitive at this stage, we will carefully discuss the choices of the step sizes  $\{\alpha_{t}\}_{t\geq 1}$ in the coming subsection. Since our results are primarily of asymptotic nature, we note that the assumption $\alpha_t \in (0,1/(\sqrt{2}{\sL}N)]$, $t \in \N$ can be relaxed and only needs be satisfied for all $t \geq t_1$ and some $t_1 \in \N$. The convergence properties derived in \Cref{Prop:convergence} and \Cref{thm:finite sum} still hold if we impose such a weaker condition on the step sizes $\{\alpha_{t}\}_{t\geq 1}$. We will exploit this straightforward observation in the next subsections. 


We close the current subsection by commenting on the slightly stronger variant of the KL inequality we utilized; see \Cref{Def:KL-property} and the discussion below it for details.  The absolute value in this variant (see \eqref{eq:kl-ineq}) is crucial for our analysis. To illustrate this observation, let us assume that we use the classical KL inequality, i.e., without taking the absolute value of ``$f(x) - f^*$''. The non-descent nature of \RR implies that the term $f(x^t) - f^*$ is not necessarily non-negative for all $t\geq 0$. Recall that the desingularizing function $\varrho$ is defined on $[0,\eta)$ for some $\eta>0$. Therefore, in this case, the last inequality in \eqref{eq:est-1} and its subsequent derivations do not hold since $\varrho'(f(x^{t-1}) - f^*)$ is not well defined. Furthermore, the stronger structural condition ``$\varrho \in \mathcal Q_\eta$'' on the desingularizing function $\varrho$ is utilized in \eqref{eq:est-1} to separate the accumulated error term $u_{t-1}$ and the function value $|f(x^{t-1})-f^*|$. In this way, the conventional KL inequality for $f$ can be applied and we can avoid imposing stringent and uncertain KL assumptions for the auxiliary terms ``$f(x^{t-1})-f^*+u_{t-1}$''.

\subsection{The Choice of Diminishing Step Sizes}\label{sec:step sizes}

In this subsection, we discuss the possible choices of the step sizes $\{\alpha_t\}_{t\geq 1}$ in \Cref{algo:gm-shuff}. Using these choices, we will also examine the constant $\sG$ defined in \cref{lemma:approx descent} and the conditions in \eqref{eq:kl-step}.

It is well-known that the non-descent nature of \RR can prevent it from converging to a stationary point if a constant step size rule is used. One of the key ingredients of our convergence analysis for \RR is the utilization of proper diminishing step sizes $\{\alpha_t\}_{t\geq 1}$, which satisfy the third condition in \eqref{eq:kl-step}.  In this subsection, we discuss the following very popular diminishing step size rule \revise{\cite{robbins1951stochastic,chung1954stochastic,bottou2018optimization,nguyen2020unified}}:
\be\label{eq:step size}
\alpha_t = \frac{\alpha}{(t+\beta)^\gamma}, \quad t = 1,2,\ldots
\ee
where $\alpha>0$, $\beta\geq 0$, and $\gamma>0$ are preset parameters. 

With such a diminishing step size rule, the first three conditions in  \eqref{eq:kl-step} can be ensured as long as 
\be \label{eq:something}
t \geq (\sqrt{2}{\sL}N\alpha)^\frac{1}{\gamma}-\beta \quad \text{and} \quad \gamma \in \left( 1/3, 1\right]. 
\ee

Before verifying the last condition in \eqref{eq:kl-step}, let us give a more explicit bound for $\sG$. Note that the two most popular choices of $\gamma$ in the diminishing step size rule \eqref{eq:step size} are $\gamma = {1}/{2}$ and $\gamma = 1$. For simplicity, let us set $\alpha = {1}/{(\sqrt{2}\sL N)}$ and $\beta = 0$. If we choose $\gamma = 1/2$, then we have $\sum_{t=1}^{\infty}  2\sL^3 N^3 \alpha_t^3 = \frac{\sqrt{2}}{2}  \sum_{t=1}^{\infty} \frac{1}{t^{3/2}} \approx  1.85$, 
which gives $\sG \leq 7 ( f(x^{0})- \bar f_{\min})$. In the case $\gamma = 1$, it follows $\sum_{t=1}^{\infty}  2\sL^3 N^3 \alpha_t^3 = \frac{\sqrt{2}}{2}  \sum_{t=1}^{\infty} \frac{1}{t^{3}} \approx  0.85$, 
which  implies $\sG \leq 3( f(x^{0})- \bar f_{\min})$.
It can be seen that $\sG$ typically has a moderate relation to the initial function value gap for both of the two popular choices of $\gamma$.

With an additional requirement  on $\gamma$,  the following lemma clarifies the last condition in \eqref{eq:kl-step} if $\varrho$ is selected from $\mathcal L$.  

\begin{lemma}\label{lemma:step size}
	Let \revise{$\theta \in [0,1)$} be given and let $\{\alpha_t\}_{t\geq 1}$ be defined according to
	\[
	\alpha_t = \frac{\alpha}{(t+\beta)^\gamma}, \quad \alpha > 0, \quad \beta \geq 0, \quad \gamma \in \left(\frac{1+\theta}{1+3\theta},1\right]. 
	\]
	Then, for all $k \geq 1$, we have
	\be\label{eq:esti-eps} 
	\frac{\underline{a}_\theta}{(k+\beta)^{(1+3\theta)\gamma-(1+\theta)}} \leq {\sum}_{t=k}^\infty \alpha_t \left[{\sum}_{j=t}^\infty\alpha_j^3\right]^{\theta} \leq \frac{\bar{a}_\theta}{(k+\beta)^{(1+3\theta)\gamma-(1+\theta)}},
	\ee 
	where $\underline{a}_\theta, \bar{a}_\theta>0$  are two positive numerical constants.
\end{lemma} 

The proof of \cref{lemma:step size} is given in \cref{app:sec-pf-step}.  \cref{lemma:step size} can be utilized to provide tight lower and upper bounds on $\sum_{t=k}^\infty \alpha_t [\varrho^\prime(u_{t-1})]^{-1}$ for any $k\geq 1$ if $\varrho \in \mathcal L$ is a {\L}ojasiewicz function. More specifically, based on \eqref{eq:step size}, \cref{thm:finite sum},  and \cref{lemma:step size}, we can formulate the following general convergence result. 

\begin{corollary} \label{cor:conv} Suppose that the assumptions \ref{A1}, \ref{B1}, and \ref{B3} are satisfied and let the sequence $\{x^t\}_{t\geq 1}$ be generated by \Cref{algo:gm-shuff}. Let us further consider the following family of step size strategies:
	\[ \alpha_t = \frac{\alpha}{(t+\beta)^\gamma}, \quad \alpha > 0, \quad \beta \geq 0, \quad \gamma \in \left(\frac{1}{2},1\right]. \]
	Then, $\{x^t\}_{t\geq 1}$ has finite length and converges to some stationary point $x^*$ of $f$. 
\end{corollary}

\begin{proof} According to \cref{lemma:uniform KL}, the mapping $f$ satisfies the KL inequality \eqref{eq:uni-kl} with a uniformized desingularizing function $\varrho \in \mathcal L$ with $\varrho(x) = c x^{1-\theta}$, $c > 0$, and $\theta \in [0,1)$. This yields $[\varrho^\prime(x)]^{-1} = x^\theta / (c(1-\theta))$ and hence, due to \eqref{eq:something} and \cref{lemma:step size} all conditions in \eqref{eq:kl-step} are satisfied for all $t$ sufficiently large if $\gamma \in ((1+\theta)/(1+3\theta),1]$. Since the KL inequality \eqref{eq:uni-kl} also holds for every desingularizing function $\bar\varrho(x) := c(1-\theta) x^{1-\bar\theta}/(1-\bar\theta)$ with 
	\[ \bar \theta \in [0,1) \quad \text{and} \quad  \bar \theta \geq \theta, \]
	(after potentially decreasing $\eta$ to $\eta < 1$) we can choose $\bar \theta$ sufficiently close to $1$ such that $\gamma \in ((1+\bar\theta)/(1+3\bar\theta),1]$ in \Cref{lemma:step size} can be relaxed to $\gamma \in (\frac12,1]$. Thus, in this situation \cref{thm:finite sum} is applicable which finishes the proof of \cref{cor:conv}. 
\end{proof}

\subsection{Convergence Rate Under the {\L}ojasiewicz Inequality} \label{sec:rate}
We have shown that the sequence $\{x^t\}_{t\geq 1}$ generated by \Cref{algo:gm-shuff} will converge to a stationary point $x^*$ of $f$. As illustrated in the last subsection, when the desingularizing functions $\varrho$ can be chosen from the class of {\L}ojasiewicz functions $\mathcal L$, then this convergence can be guaranteed for a large family of step sizes $\alpha_t = \mathcal O(1/t^\gamma)$ with $\gamma \in (\frac12,1]$. In such case, it is also typically possible to obtain the convergence rates of descent algorithms by applying the standard KL inequality-based analysis framework; see, e.g., \cite[Theorem 2]{AttBol09}.   In this subsection, our goal is to establish the convergence rate of \RRp, relying on the {\L}ojasiewicz inequality and the corresponding KL exponent $\theta$ of the stationary point $x^*$; see \eqref{eq:Lojasiewicz ineq} for a concrete definition.

Throughout this subsection,  the  desingularizing function $\varrho$  is a {\L}ojasiewicz function selected from $\mathcal L$ (see \eqref{eq:Lojasiewicz function}), i.e., it  satisfies $\varrho(x) = cx^{1-\theta}$, where  $c>0$ and $\theta \in [0,1)$. Then, the KL inequality \eqref{eq:kl-ineq} and its uniformized version \eqref{eq:uni-kl} simplify to the following {{\L}ojasiewicz inequality}:
\be\label{eq:Lojasiewicz ineq}
|f(x)-f^*|^\theta \leq c(1-\theta) \|\nabla{f}(x)\|,
\ee
where $\theta$ is the KL exponent of $f$ at $x^*$. 

Before presenting our convergence rate analysis, let us restate two results that have been shown in \cite[Lemma 4 and 5]{Pol87}, which are crucial to our analysis. 
\begin{lemma} \label{lemma:rate} Let $\{z_k\}_{k\geq 1}$ be a non-negative sequence and let $b \geq 0$, $d, p, q > 0$, $s \in (0,1)$, and $\tau > s$ be given constants. 
	\begin{enumerate} [label=(\alph*),topsep=0pt,itemsep=0ex,partopsep=0ex]
		\item Suppose that the sequence $\{z_k\}_{k\geq 1}$ satisfies 
		\[ 
		z_{k+1} \leq \left(1- \frac{q}{k+b} \right) z_k + \frac{d}{(k+b)^{p+1}}, \quad \forall \ k\geq 1.
		\]
		Then, if $q > p$, it holds that $ z_k \leq \frac{d}{q-p}\cdot (k+b)^{-p} + o((k+b)^{-p})$ for all sufficiently large $k$. 
		\item Suppose that the sequence $\{z_k\}_{k\geq 1}$ satisfies 
		\[ z_{k+1} \leq \left(1- \frac{q}{(k+b)^s} \right) z_k + \frac{d}{(k+b)^{\tau}}, \quad \forall~k \geq 1. \]
		Then, it follows $z_k \leq \frac{d}{q} \cdot (k+b)^{s-\tau} + o((k+b)^{s-\tau})$.
	\end{enumerate}
\end{lemma}

We note that \cref{lemma:rate} is a slight extension of the original results in \cite{Pol87} as it contains the additional scalar $b \geq 0$. Since the proof of \cref{lemma:rate} is identical to the ones given in \cite{Pol87}, we will omit an explicit verification here. We now present our second main result. 

\begin{theorem}\label{thm:convergence rate} Suppose that the conditions stated in \cref{cor:conv} are satisfied and let us consider the following family of step sizes  $\{\alpha_t\}_{t\geq 1}$: 
	\be \label{eq:final-steps} \alpha_t = \frac{\alpha}{(t+\beta)^\gamma}, \quad \alpha > 0, \quad \beta \geq 0, \quad \gamma \in \left(\frac{1}{2},1\right]. \ee
	Then, the sequence $\{x^t\}_{t\geq 1}$ converges to some stationary point $x^*$ of $f$. Let $\theta \in [0,1)$ \revise{and $c > 0$} denote the KL exponent of $f$ at $x^*$ \revise{and the corresponding KL constant, respectively}. Then, for all sufficiently large $t$, we have
	\[
	\|x^t - x^*\| = 
	\begin{cases} \mathcal O(t^{-(2\gamma-1)}) & \text{if} \ 0\leq \theta < \frac{\gamma}{3\gamma-1}, \\ 
	\mathcal O(t^{-\frac{(1-\theta)(1-\gamma)}{2\theta-1}}) & \text{if} \ \frac{\gamma}{3\gamma-1} \leq  \theta <1,
	\end{cases} \quad \text{where} \ \gamma \in \left(\frac12,1\right).
	\]
	Moreover, in the  case $\theta \in [0, \frac12]$ and $\gamma = 1$, if we set \revise{$\alpha > \frac{8c^2}{N}$}, then it follows $\|x^t - x^*\| = \mathcal O(t^{-1})$.
\end{theorem}

\begin{proof}
For simplicity and without loss of generality, we assume $\beta = 0$ throughout the proof. Convergence of $\{x^t\}_{t \geq 1}$ was established in \cref{cor:conv}. Furthermore, as shown in the proof of \cref{thm:finite sum}, the sequence $\{f(x^t)+u_t\}_{t \geq 1}$ monotonically decreases and converges to $f^*$. Hence, there exists $t_1 \geq 1$ such that $f(x^t)-f^*+u_t < 1$ for all $t \geq t_1$. Moreover, due to $\alpha_t \to 0$, we may increase $t_1$ so as to ensure $\alpha_t \leq 1/(\sqrt{2}\sL N)$ for all $t \geq t_1$. Throughout this proof, we always assume that $t \geq \max\{t_0, t_1\} =: t_2$ is sufficiently large, where $t_0$ is defined in the proof of \cref{thm:finite sum}. 

Let the associated desingularizing mapping $\varrho$ for the limit point $x^*$ be given by $\varrho(x) = cx^{1-\theta}$.  As in the proof of \cref{cor:conv}, the {\L}ojasiewicz inequality \eqref{eq:Lojasiewicz ineq} also holds for every exponent $[0,1) \ni \tau \geq \theta$. Hence, we can work with the following adjusted exponent $\tau$ and desingularizing function $\varrho_\tau$:
\[ \tau = \tau(\zeta) := \max\left\{\frac{1-\gamma}{3\gamma-1},\theta\right\} + \zeta \quad \text{and} \quad \varrho_\tau(x) = \frac{c(1-\theta)}{1-\tau}x^{1-\tau}, \]
where $\zeta \geq 0$ is chosen such that $\theta \leq \tau <1$ and $\tau > \frac{1-\gamma}{3\gamma-1}$. Since the desingularizing function $\varrho_\tau$ is a {\L}ojasiewicz function chosen from $\mathcal L$, condition \eqref{eq:qrho} holds with $C_{\varrho_\tau} = 1$ (see the discussion after \Cref{Def:KL-property}).  Thus, setting $\varrho\equiv \varrho_\tau$, $\vartheta = \frac12$, and $T = \infty$ in \eqref{eq:festi-kl},  it follows
\be\label{eq:rate-est-1}
\frac{N}{2} \sum_{t=k}^{\infty} \alpha_{t} \|\nabla f(x^{t-1})\| + \bar\vartheta \sum_{t=k}^{\infty} \|x^{t}-x^{t-1}\| \leq 2\delta_{k-1}  + \frac{N}{2} \sum_{t=k}^{\infty} \alpha_{t} [\varrho_\tau^\prime(u_{t-1})]^{-1}
\ee
for every $k \geq t_2 + 1$. 
%
%
%
Using the adjusted desingularizing function $\varrho_\tau$ and the definition of $\delta_t$, we can obtain
\be\label{eq:rate-est-2}
\begin{aligned}
    \delta_t &= \varrho_\tau(f(x^t)-f^* + u_t) \leq \frac{c(1-\theta)}{1-\tau}  (|f(x^t)-f^*| + u_t)^{1-\tau}  \\
	& = \frac{c(1-\theta)}{1-\tau} [ c(1-\theta) \cdot [\varrho_\tau^\prime(|f(x^t)-f^*| + u_t)]^{-1} ]^{\frac{1-\tau}{\tau}} \\ 
	&  \leq (c(1-\theta))^{\frac{1}{\tau}}(1-\tau)^{-1}[ \|\nabla f(x^t)\| + [\varrho_\tau^\prime(u_t)]^{-1}]^{\frac{1-\tau}{\tau}}, 
\end{aligned}
\ee
for all $t \geq t_2$, where the \revise{second} line is due to the special structure of the desingularizing function $\varrho_\tau$ and the last line follows from  condition \eqref{eq:qrho} and the {\L}ojasiewicz inequality \eqref{eq:Lojasiewicz ineq}. Combining \eqref{eq:rate-est-1} and \eqref{eq:rate-est-2} yields 
\be
\begin{aligned}  \label{eq:rate-est}
   &\sum_{t=k}^\infty \alpha_t [\|\nabla f(x^{t-1})\| + [\varrho_\tau^\prime(u_{t-1})]^{-1}] + \frac{2\bar\vartheta}{N} \sum_{t=k}^\infty \|x^t - x^{t-1}\|  \\ 
   &\hspace{4ex} \leq \frac{4(c(1-\theta))^{\frac{1}{\tau}}}{N(1-\tau)}  [ \|\nabla f(x^{k-1})\| + [\varrho_\tau^\prime(u_{k-1})]^{-1}]^{\frac{1-\tau}{\tau}} + 2 \sum_{t=k}^{\infty} \alpha_{t} [\varrho_\tau^\prime(u_{t-1})]^{-1} \end{aligned}
\ee
for every $k \geq t_2 + 1$. Upon setting $\varepsilon_k := 2 \sum_{t=k}^{\infty} \alpha_{t} [\varrho_\tau^\prime(u_{t-1})]^{-1}$ and
\begin{align*}
	\Gamma_k &:= {\sum}_{t=k}^\infty \alpha_t [\|\nabla f(x^{t-1})\| + [\varrho_\tau^\prime(u_{t-1})]^{-1}] + \frac{2\bar\vartheta}{N} {\sum}_{t=k}^\infty \|x^t - x^{t-1}\|,
\end{align*}
the inequality \eqref{eq:rate-est} can be rewritten as
\be \label{eq:kl-to-gam} 
\Gamma_k \leq \frac{4(c(1-\theta))^{\frac{1}{\tau}}}{N(1-\tau)}  \left [ \frac{\Gamma_{k}-\Gamma_{k+1}}{\alpha_k} \right]^{\frac{1-\tau}{\tau}} + \varepsilon_k. 
\ee

To apply \cref{lemma:step size} to further bound $\varepsilon_k$, we require the exponent $\tau \in [0,1)$ and the step size parameter $\gamma \in (\frac12,1]$ to satisfy the condition
\[ \gamma > \frac{1+\tau}{1+3\tau} \quad \iff \quad \tau > \frac{1-\gamma}{3\gamma-1}. \]
In this case,  recalling $u_{t-1} = 2\sG \sL^3 N^3 \sum_{j=t}^\infty \alpha_j^3$, we can derive the following bound for the error term $\varepsilon_k$:
\be \label{eq:err-eps} \varepsilon_k \leq 2^{1+\tau} ({\sf G}{\sf L}^3N^3)^\tau {\sum}_{t=k}^\infty \alpha_t \left[{\sum}_{j=t}^\infty \alpha_j^3 \right]^\tau \leq \frac{2^{1+\tau} ({\sf G}{\sf L}^3N^3)^\tau\bar a_\tau}{k^{\nu}} =: \frac{A_\tau}{{k^{\nu}}}, \ee
where $\nu := (1+3\tau)\gamma-(1+\tau)$ and $\bar a_\tau$ is the constant defined in \cref{lemma:step size}. 

By the triangle inequality and the fact that $\{x^t\}_{t\geq 1}$ converges to $x^*$, we have 
\[
\|x^{k-1} -x^*\|\leq {\sum}_{t=k}^\infty \|x^t - x^{t-1}\| \leq \frac{N}{2\bar\vartheta} \cdot \Gamma_k.
\]
Thus, in order to establish the rate of $\|x^k -x^*\| \rightarrow 0$, it suffices to derive the rate of $\Gamma_{k+1}  \rightarrow 0$ based on the recursion \eqref{eq:kl-to-gam}. In the following, we will prove the convergence rates depending on the value of the exponent $\tau$. 

\textbf{Case 1:} $\tau =\frac12$.  
In this case, the estimates in \eqref{eq:kl-to-gam} and \eqref{eq:err-eps} simplify to
\begin{align*}
	\Gamma_{k+1}&\leq\left[1-\frac{N\alpha_k}{8(c(1-\theta))^2}\right]\Gamma_{k}+\frac{N\alpha_k\varepsilon_k}{8(c(1-\theta))^2} \\
	&\leq \left[1-\frac{N\alpha}{8(c(1-\theta))^2}\frac{1}{k^\gamma} \right]\Gamma_{k} + \frac{N\alpha A_{1/2}}{8(c(1-\theta))^2} \frac{1}{k^{(7\gamma-3)/2}}.
\end{align*}
%
If $(1+\tau)/(1+3\tau) = 3/5<\gamma<1$, the rate is ${(7\gamma - 3)}/{2}-\gamma={(5\gamma-3)}/{2}$ by applying \Cref{lemma:rate} (b) and we obtain
\[  
\|x^t - x^*\| = \mathcal O(t^{-\frac{5\gamma-3}{2}}).
\]
Moreover, if $\gamma = 1$ and we set \revise{$\alpha > \frac{8c^2}{N} \geq \frac{8(c(1-\theta))^2}{N}$}, then \Cref{lemma:rate} (a) yields $\|x^t - x^*\| = \mathcal O(t^{-1})$.

\textbf{Case 2:} $\tau \in (\frac12,1)$, $\gamma \neq 1$. In this case, we have $\frac{1-\tau}{\tau} = \frac{1}{\tau} - 1 \in (0,1)$. Hence, invoking Minkowski's inequality---i.e., $(|a|+|b|)^p \leq 2^{p-1}(|a|^p + |b|^p)$, $p > 1$, $a,b \in \R$---in \eqref{eq:kl-to-gam}, we obtain
\be \label{eq:esti-nr5} 
\Gamma_k^{\frac{\tau}{1-\tau}} \leq C_\tau \left[ \frac{\Gamma_k-\Gamma_{k+1}}{\alpha_k} \right] + 2^{\frac{2\tau-1}{1-\tau}} \varepsilon_k^{\frac{\tau}{1-\tau}},
\ee
where $C_\tau:= 2^{\frac{4\tau-1}{1-\tau}}(c(1-\theta))^{\frac{1}{1-\tau}} / (N(1-\tau))^{\frac{\tau}{1-\tau}}$. We now discuss three different sub-cases.

\textbf{Sub-Case 2.1:} $\frac{1-\gamma}{3\gamma-1}<\tau < \frac{\gamma}{3\gamma-1}$. Due to $\frac{\tau}{1-\tau} > 1$, the function $x \mapsto h_\tau(x) := x^{\frac{\tau}{1-\tau}}$ is convex when $x\geq 0$, i.e., we have 
\[h_\tau(y) - h_\tau(x) \geq \frac{\tau}{1-\tau} \cdot x^{\frac{2\tau-1}{1-\tau}} (y-x) \quad \forall~x,y\in\R_{+}. \]
Let $\sigma > 0$ be an arbitrary constant. Rearranging the terms in \eqref{eq:esti-nr5} and using \eqref{eq:err-eps} and the convexity of $h_\tau$  yields 
\begingroup
\allowdisplaybreaks
\begin{align} 
	\nonumber \Gamma_{k+1} &\leq \Gamma_k - \frac{\alpha_k}{C_\tau} \Gamma_{k}^{\frac{\tau}{1-\tau}} + \frac{2^{\frac{2\tau-1}{1-\tau}}}{C_\tau} \alpha_k \varepsilon_k^{\frac{\tau}{1-\tau}} \\
	\nonumber & \leq \Gamma_k - \frac{\alpha_k}{C_\tau} [ h_\tau(\Gamma_k) - h_\tau(\sigma / k^\nu)] - \frac{\sigma^{\frac{\tau}{1-\tau}}}{C_\tau} \frac{\alpha_k}{k^{\frac{\nu\tau}{1-\tau}}} + \frac{2^{\frac{2\tau-1}{1-\tau}}A_\tau^{\frac{\tau}{1-\tau}}}{C_\tau} \frac{\alpha_k}{k^{\frac{\nu\tau}{1-\tau}}} \\ 
	\label{eq:big-estimate} &\leq \Gamma_k - \frac{\alpha_k}{C_\tau} \frac{\tau}{1-\tau} \left(\frac{\sigma}{k^\nu}\right)^{\frac{2\tau -1}{1-\tau}} \left[\Gamma_k - \frac{\sigma}{k^\nu}\right] + \left(2^{\frac{2\tau-1}{1-\tau}}A_\tau^\frac{\tau}{1-\tau}-\sigma^\frac{\tau}{1-\tau}\right) \frac{\alpha_k}{C_\tau k^{\frac{\nu\tau}{1-\tau}}} \\
	 \nonumber & = \left [ 1 - \frac{\sigma^\frac{2\tau-1}{1-\tau}\tau}{C_\tau(1-\tau)} \frac{\alpha}{k^{\frac{2\tau-1}{1-\tau}\cdot \nu+\gamma}} \right] \Gamma_k + \left[{2^{\frac{2\tau-1}{1-\tau}}\frac{A_\tau^\frac{\tau}{1-\tau}}{C_\tau}+\frac{2\tau-1}{1-\tau}}\frac{\sigma^\frac{\tau}{1-\tau}}{C_\tau}\right] \frac{\alpha}{k^{\frac{\tau}{1-\tau}\cdot\nu+\gamma}} . 
\end{align}
\endgroup
Here, let us mention that the application of the convexity of $h_\tau(x)$ is motivated by the proof of \cite[Lemma 2.2]{HuLiYu20}. Recalling $\nu = (1+3\tau)\gamma-(1+\tau)$ and using some simple algebraic manipulations, it can be shown that the conditions $\frac{2\tau-1}{1-\tau}\cdot \nu + \gamma < 1$ is equivalent to  $\tau < \frac{\gamma}{3\gamma-1}$.  
Consequently, in this sub-case, \Cref{lemma:rate} (b) is applicable and we have $ \|x^t - x^*\| = \mathcal O(t^{-\nu}) = \mathcal O(t^{-[(1+3\tau)\gamma - (1+\tau)]})$. 

\textbf{Sub-Case 2.2:} $\tau = \frac{\gamma}{3\gamma-1}$. In this case, we obtain $\frac{2\tau-1}{1-\tau}\cdot \nu + \gamma = 1$ and $\frac{\tau}{1-\tau}\cdot \nu + \gamma = 1 + \frac{1-\tau}{3\tau-1}$. Thus, if we choose the constant $\sigma$ in sub-case 2.1 such that $\sigma^\mu > {C_\tau(1-\tau)^2}/[(3\tau-1)\tau\alpha]$ where $\mu = (2\tau-1)/{(1-\tau)}$, we can apply \Cref{lemma:rate} (a) to \eqref{eq:big-estimate} to infer
\[
\|x^t - x^*\| =  \mathcal O(t^{-\frac{1-\tau}{3\tau-1}}). 
\]

\textbf{Sub-Case 2.3:} $\tau > \frac{\gamma}{3\gamma-1}$. Let $\sigma > \frac{C_\tau(1-\tau)^2(1-\gamma)}{(2\tau-1)\tau\alpha}$ be a given constant and let us set $\beta_k := (\sigma /k^{1-\gamma})^\frac{1-\tau}{2\tau-1}$. By repeating the steps of sub-case 2.1, we obtain
\begin{align*}
	\Gamma_{k+1} &\leq \Gamma_k - \frac{\alpha_k}{C_\tau} [ h_\tau(\Gamma_k) - h_\tau(\beta_k)] - \frac{\alpha_k}{C_\tau} h_\tau(\beta_k) + \frac{2^{\frac{2\tau-1}{1-\tau}}}{C_\tau} \alpha_k \varepsilon_k^{\frac{\tau}{1-\tau}} \\ 
	&\leq \Gamma_k - \frac{\alpha_k}{C_\tau} \frac{\tau}{1-\tau} \frac{\sigma}{k^{1-\gamma}} \left[\Gamma_k - \left(\frac{\sigma}{k^{1-\gamma}}\right)^{\frac{1-\tau}{2\tau -1}}\right] +  \frac{2^{\frac{2\tau-1}{1-\tau}}A_\tau^\frac{\tau}{1-\tau}}{C_\tau} \cdot \frac{\alpha}{k^{\frac{\tau}{1-\tau}\cdot\nu+\gamma}}\\
	& = \left [ 1 - \frac{\tau\sigma}{C_\tau(1-\tau)} \frac{\alpha}{k} \right] \Gamma_k + \frac{\tau\sigma^\frac{\tau}{1-\tau}}{C_\tau(1-\tau)} \frac{\alpha}{k^{\frac{\tau}{2\tau-1}\cdot(1-\gamma)+\gamma}} + \frac{2^{\frac{2\tau-1}{1-\tau}}A_\tau^\frac{\tau}{1-\tau}}{C_\tau} \cdot \frac{\alpha}{k^{\frac{\tau}{1-\tau}\cdot\nu+\gamma}}. 
\end{align*}
Next, we discuss the last two terms of this estimate to determine the leading order. Using $\frac12 <\tau < 1$, $\gamma > \frac12$, and some algebraic manipulations, we can show
\begin{align*} & \hspace{-4ex} \frac{\tau}{2\tau-1} (1-\gamma) < \frac{\tau}{1-\tau} [(1+3\tau)\gamma - (1+\tau)] \; \\
	&\iff \; 2\tau^2 - \gamma(1-\tau) < (2\tau-1)(1+3\tau)\gamma \; \iff \; \tau > \frac{\gamma}{3\gamma-1}.  \end{align*}
Hence, there exists $C_\tau^\prime(\sigma)$ such that we have
\[ \Gamma_{k+1} \leq  \left [ 1 - \frac{\tau\sigma}{C_\tau(1-\tau)} \frac{\alpha}{k} \right] \Gamma_k + \frac{C_\tau^\prime(\sigma)}{k^{\frac{\tau}{2\tau-1}\cdot(1-\gamma)+\gamma}} \]
for all $k$ sufficiently large. Moreover, due to $\frac{\tau}{2\tau-1}(1-\gamma)+\gamma = 1 + \frac{(1-\tau)(1-\gamma)}{2\tau-1}$, the choice of $\sigma$ ensures that \Cref{lemma:rate} (a) is applicable. This establishes
\[ \|x^t - x^*\| = \mathcal O(t^{-\frac{(1-\tau)(1-\gamma)}{2\tau-1}}). \]
Finally, we express the obtained rates in terms of the initial KL exponent $\theta$. We first notice that the mapping $\psi_\gamma : (\frac12,1) \to (0,1)$,  
\[\psi_\gamma(\tau) := \begin{cases} (3\gamma-1)\tau+\gamma-1 & \text{if } \frac{1-\gamma}{3\gamma-1} < \tau < \frac{\gamma}{3\gamma-1}, \\ \frac{1-\tau}{3\tau-1} & \text{if } \tau = \frac{\gamma}{3\gamma-1}, \\ \frac{(1-\tau)(1-\gamma)}{2\tau -1} & \text{if } \tau > \frac{\gamma}{3\gamma-1}, \end{cases} \quad \quad \gamma \in \left(\frac12,1\right)\]
is continuous on $(\frac12,1)$, increasing on $(\frac12,\frac{\gamma}{3\gamma-1}]$, and decreasing on $[\frac{\gamma}{3\gamma-1},1)$. Hence, in the case $\theta \geq \frac{\gamma}{3\gamma-1}$, we have $\tau = \theta + \zeta$ and the optimal rate is attained for $\zeta = 0$. Furthermore, in the case $0 \leq \theta < \frac{\gamma}{3\gamma-1}$,  the rate is maximized if we choose  $\zeta$ such that $\tau = {\gamma}/{(3\gamma-1)}$, which yields $\psi_\gamma(\gamma/(3\gamma-1)) = 2\gamma - 1$. This verifies the rates formulated in \Cref{thm:convergence rate} for all pairs $(\theta,\gamma) \in [0,1) \times (\frac12,1)$.
The  case $\theta \leq \frac12$ and $\gamma = 1$ is fully covered by \textbf{Case 1} (by properly choosing $\zeta$ such that $\tau=\frac12$). Therefore, all cases stated in \Cref{thm:convergence rate} are established.

The more general case $\beta \neq 0$ can be discussed in exactly the same way  by invoking \cref{lemma:rate} with $b \equiv \beta$.
\end{proof}

Compared to other standard KL results \cite{AttBol09,BolSabTeb14}, the derivation of the convergence rates of \RR is more involved due to the intricate interaction between the KL exponent $\theta$, \revise{the non-descent nature of \RRp,} and the dynamics of the diminishing step sizes $\{\alpha_t\}_{t\geq 1}$. \revise{More specifically, our rate recursion \eqref{eq:kl-to-gam} involves  the error $\varepsilon_k$ and the diminishing step sizes in a complicated way. This causes additional technical difficulties --- especially for the case $\theta \in (\frac12,1)$. In fact, the obtained recursion \eqref{eq:esti-nr5} differs significantly from the standard one \cite[Equation (13)]{AttBol09} and hence, we cannot follow the standard steps after \cite[Equation (13)]{AttBol09} to derive the rate results.} 
It can be observed from \Cref{thm:convergence rate} that when the KL exponent $\theta\in[0,\frac12]$ and $\gamma = 1$, the convergence rate of \RR is $\cO(t^{-1})$, matching that of the strongly convex setting \cite{gurbu2019}. When $\theta\in (\frac12,1)$ and $\gamma \in (\frac12,1)$, the convergence rate is $\cO(t^{-q})$ with $q\in (0,1)$ depending on $\theta$ and $\gamma$.  In \cref{fig:rate}, we illustrate the obtained rate results graphically.

\begin{figure}[t]
	\label{fig:rate}
	\centering
	\includegraphics[height=4.95cm]{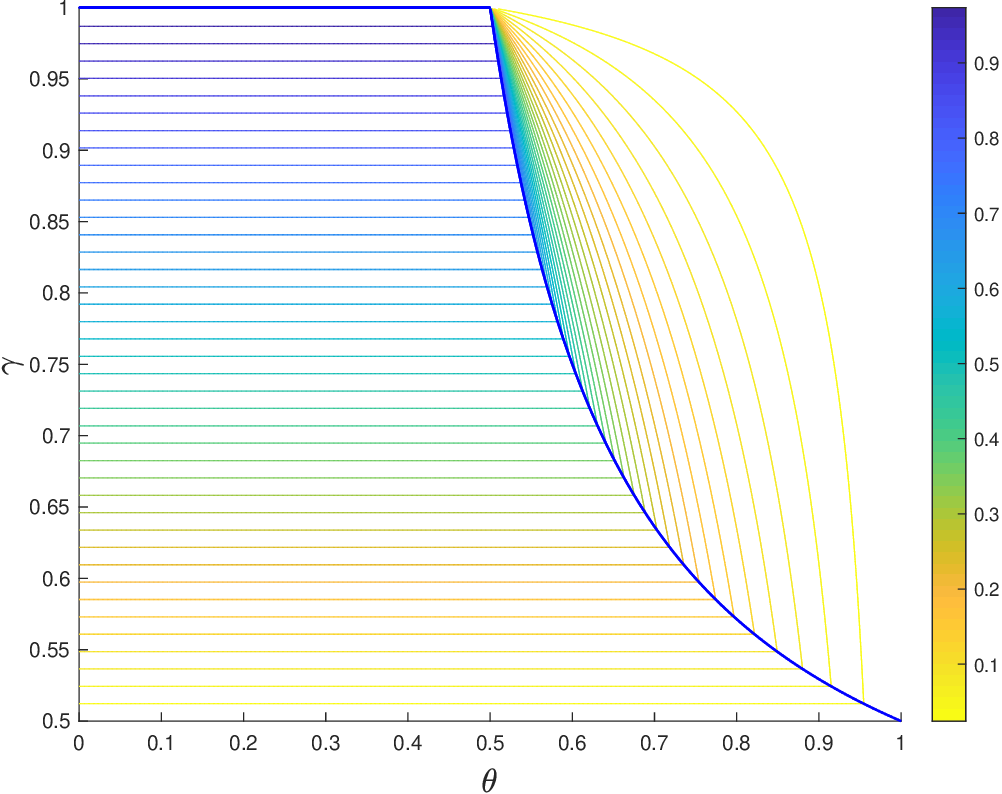} \,
	\includegraphics[height=4.95cm,trim=0ex -1.2ex 0ex -1ex, clip]{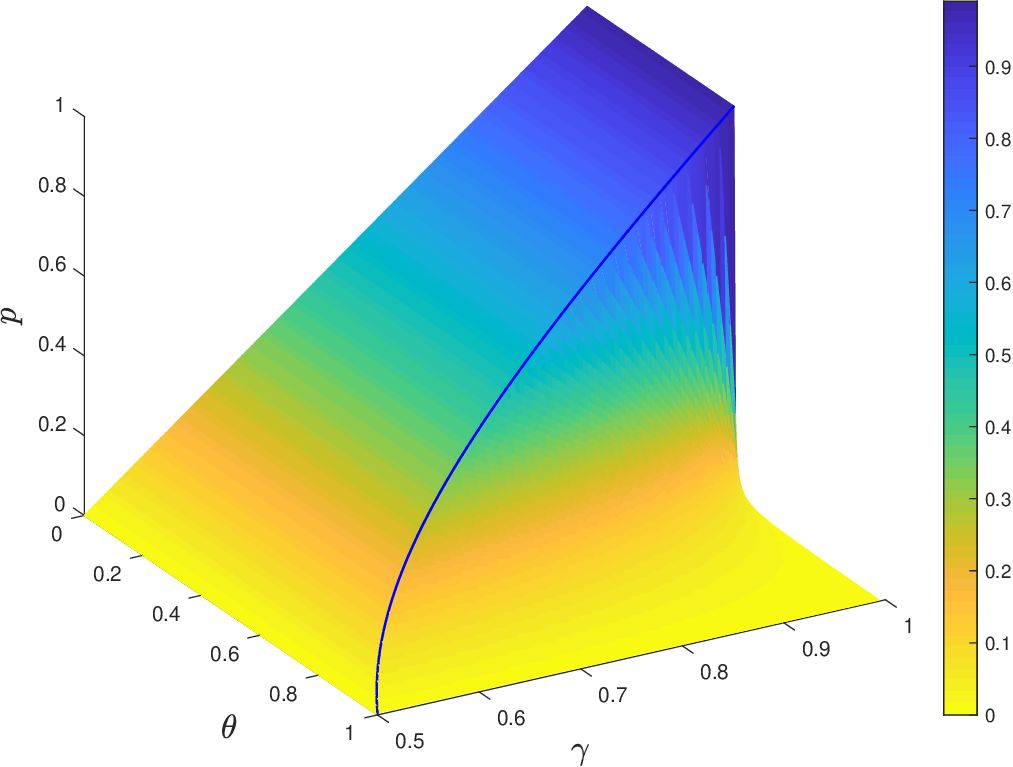} 
	\caption{Contour and surface plot of the rates $\mathcal O(t^{-p})$ obtained in \cref{thm:convergence rate} as a multifunction of the step size parameter $\gamma \in (\frac12,1]$ and the KL exponent $\theta \in [0,1)$. The blue (transition) line depicts the rate along the set $\{(\theta,1): \theta \in [0,\frac12)\} \cup \{(\theta,\gamma): \theta = \gamma/(3\gamma-1)\}$. }
\end{figure}

\subsection{Extension to Almost Sure Results} \label{sec:sto-sure}

As we have mentioned, \RR naturally is a stochastic optimization algorithm. Thus, it is paramount to generalize the convergence results derived for a single trajectory and to discuss convergence in a more rigorous stochastic framework. Fortunately, many of our results do not depend on the specific choice of the random permutations $\{\sigma^t\}_{t\geq 1}$ and can be easily transferred to the fully stochastic setting. 

In the following, we briefly specify the underlying stochastic model for \RR which allows us to describe the stochastic components of \Cref{algo:gm-shuff} in a unified way. Based on the structure of \Cref{algo:gm-shuff}, we can define a suitable sample space $\Omega$ via $\Omega := \Lambda^\infty := \{\omega = (\omega_1,\omega_2,\omega_3,\dots): \omega_i \in \Lambda \, \text{for all} \, i\}$, where $\Lambda$ is defined in \eqref{eq:permutations}. Each sample $\omega \in \Omega$ then corresponds to a specific sequence of (randomly chosen) permutations as generated, e.g., by a single run of \Cref{algo:gm-shuff}. Similar to infinite coin toss models, this sample space can be naturally equipped with a $\sigma$-algebra $\mathcal F$ and a probability function $\mathbb P$ to form a proper probability space $(\Omega,\mathcal F,\Prob)$. 

Given this stochastic model, let us now consider a general stochastic process of iterates $\{X^t\}_{t \geq 1}$ generated by \RRp. So far the previous analysis was concerned with the investigation of a single trajectory $\{X^t(\omega)\}_{t \geq 1} \equiv \{x^t\}_{t \geq 1}$ for some fixed $\omega \in \Omega$. Since the proof of \cref{Prop:convergence} is purely deterministic, we can immediately infer:
\begin{itemize}
	\item Under assumption \ref{A1} and the step size condition \eqref{eq:ass-step}, the stochastic process $\{f(X^t)\}_{t \geq 1}$ surely converges to some random variable $F^* : \Omega \to \R$ and we have $\lim_{t \to \infty} \|\nabla f(X^t)\| = 0$ surely (i.e., for all $\omega \in \Omega$). 
\end{itemize}
Generalizing the KL inequality-based results derived in \Cref{sec:sequential} requires some more subtle changes. In particular, the set of accumulation points $\mathcal C$ defined in \eqref{eq:limit point set} depends on the selected sample $\omega \in \Omega$ and the trajectories of $\{X^t\}_{t \geq 1}$ might generally converge to different limit points. Hence, we need to work with proper stochastic extensions of the KL conditions formulated in \cref{Assumption:2}. Let us first formally introduce the set of accumulation points $\mathcal C$ of $\{X^t\}_{t \geq 1}$ as the set-valued multifunction 
\[ \mathcal C : \Omega \rightrightarrows \Rn, \quad  \mathcal C(\omega) := \{ x \in \Rn: \exists~\text{a subsequence} \; \{t_\ell\}_\ell \; \text{s.t.} \;  X^{t_\ell}(\omega) \to x \text{ as } \ell \to \infty\}. \]
We can then consider the following stochastic versions of condition \ref{B1} and \ref{B3}.

\begin{assumption}
	\label{Assumption:3} 
	We assume the following conditions:
	\begin{enumerate}[label=\textup{\textrm{(C.\arabic*)}},topsep=0pt,itemsep=0ex,partopsep=0ex]
		\item \label{C1} The stochastic process $\{X^t\}_{t\geq 1}$ is bounded almost surely, i.e., $\sup_t \|X^t\| < \infty$ almost surely.
		\item \label{C2} There exists an event $\mathcal K \in \mathcal F$ with $\Prob(\mathcal K) = 1$ such that the KL inequality holds for every $\bar x \in \bigcup_{\omega \in \mathcal K}\mathcal C(\omega)$ and each respective desingularizing function $\varrho$ can be selected from the class of {\L}ojasiewicz functions $\mathcal L$. 
	\end{enumerate}
\end{assumption}

Notice that assumption \ref{C2} still holds if $f$ is subanalytic or semialgebraic. Furthermore, since the sequence $\{f(X^t)\}_{t\geq1}$ converges surely to some $F^*$, condition \ref{C1} is satisfied when $f$ is coercive or when it has bounded level sets. Under \ref{C1} and \ref{C2}, we can now derive an almost sure variant of the convergence results established in \cref{cor:conv} and \cref{thm:convergence rate}.

\begin{theorem}\label{thm:sto-rate} Suppose that the assumptions \ref{A1}, \ref{C1}, and \ref{C2} are satisfied and let the step sizes $\{\alpha_t\}_{t\geq 1}$ be defined as in \eqref{eq:final-steps} for suitable choices of the parameters $\alpha$, $\beta$, and $\gamma$.
	%
	%
	Then, the stochastic process $\{X^t\}_{t\geq 1}$ almost surely converges to a $\mathrm{crit}(f)$-valued random vector $X^* : \Omega \to \mathrm{crit}(f)$. Furthermore, let $\Theta : \Omega \to [0,1)$ denote the function that maps each sample $\omega \in \mathcal K$ to the KL exponent of $f$ at $X^*(\omega)$ and let $C : \Omega \to \R_+$ be the corresponding function of KL constants. Then, the event 
	\[ \mathcal G := \left\{ \omega \in \Omega: {\limsup}_{t \to \infty} \; t^{\Psi(\omega)} \cdot \|X^t(\omega) - X^*(\omega)\| < \infty \right\} \]
	occurs almost surely, where the rate function $\Psi : \Omega \to \R_+$ is given by 
	\be \label{eq:psi} \Psi(\omega) := \begin{cases} 2\gamma-1 & \text{if } 0 \leq \Theta(\omega) < \frac{\gamma}{3\gamma-1} \; \text{and} \; \gamma \in (\frac12,1), \\ \frac{(1-\Theta(\omega))(1-\gamma)}{2\Theta(\omega)-1} & \text{if } \frac{\gamma}{3\gamma-1} \leq \Theta(\omega) <1  \; \text{and} \; \gamma \in (\frac12,1), \\ 1 & \text{if } \Theta(\omega) \in [0,\frac12], \; \gamma = 1,  \; \text{and} \; \revise{\alpha > \frac{8C^2(\omega)}{N}}, \\ 0 & \text{otherwise.} 
	\end{cases} \ee
	%
	%
\end{theorem}

\begin{proof} Assumptions \ref{C1} and \ref{C2} imply that there is an event $\mathcal E \in \mathcal F$ with $\Prob(\mathcal E) = 1$ such that $\{X^t(\omega)\}_{t\geq 1}$ is bounded for all $\omega \in \mathcal E$ and the {\L}ojasiewicz inequality holds for every $\bar x \in \mathcal C(\omega)$ and all $\omega \in \mathcal E$. Hence, the results in \cref{cor:conv} and \cref{thm:convergence rate} can be applied to all trajectories $\{X^t(\omega)\}_{t\geq 1}$, $\omega \in \mathcal E$ and we can infer $\mathcal E \subseteq \mathcal G$, which proves \cref{thm:sto-rate}. \end{proof}


\section{An Informal Analysis Framework and Reshuffled Proximal Point Method}

In \Cref{sec:convergence analysis}, we established the strong limit-point convergence results for \RR under the KL inequality. One remarkable feature of our results is that \RR is a non-descent method and the existing standard KL  analysis framework is not applicable. In this section, we summarize the main steps and core ideas in an analysis framework. As an immediate application of such a framework,  we show that the reshuffled proximal point method also shares the same convergence results as those of {\sf RR}.

\subsection{An Informal Analysis Framework}\label{sec:framework}
We consider the general problem 
\[
\minimize_{x\in \R^n} \ f(x), 
\]
where the function $f: \R^n \rightarrow \R$ is smooth and bounded from below. In order to solve this problem, suppose we further have access to a generic (\emph{non-descent}) algorithm $\cal A$ with \emph{diminishing step sizes}  $\{\alpha_t\}_{t\geq1}$, which iterates as
\[
x^{t} = \mathcal{A} (x^{t-1}, \alpha_t), \quad t = 1,2,\ldots
\]

Typical requirements on the diminishing step sizes are 
\[
\alpha_t >0, \quad {\sum}_{t=1}^{\infty} \alpha_t = \infty, \quad \text{and} \quad {\sum}_{t=1}^{\infty} \alpha_t^p < \infty
\]
for some $p\geq 2$. 
The goal is to establish strong limit-point convergence and derive the corresponding rate of convergence for algorithm $\mathcal{A}$. Our analysis in \Cref{sec:convergence analysis} consists of three main steps, which we summarize below in a simplified framework. For ease of exposition, we consider fixed permutations $\{\sigma^t\}_{t\geq 1}$ without randomness.

\begin{enumerate}
	\item[({\sf A})] \textbf{Approximate descent property.}  Based on the algorithmic properties of $\mathcal A$, show that
	\[
	f({x}^t)\leq f({x}^{t-1})- \kappa_1 \alpha_t \|\nabla{f}({x}^{t-1})\|^2- \frac{\kappa_2}{\alpha_t}\|{x}^{t}-{x}^{t-1}\|^2 +\kappa_3 \alpha_t^q,
	\]
	where $\kappa_1,\kappa_3 >0$, $\kappa_2 \geq 0$ and $q\geq p \geq 2$ are some constants.
	
	\item[({\sf B})] \textbf{Weak convergence.} Based on the approximate descent property ({\sf A}), establish weak convergence of $\mathcal A$, i.e., $\lim_{t \to \infty} \|\nabla f(x^t)\| = 0$.  Consequently, any accumulation point of $\{x^t\}_{t\geq 1}$ is guaranteed to be a stationary point of $f$. 
\end{enumerate}

As mentioned, the sufficient decrease property is a ubiquitous ingredient of the standard KL analysis framework. With the help of this  algorithmic property, it is quite straightforward to establish the weak convergence result formulated in ({\sf B}); see, e.g., \cite[Section 3.2]{BolSabTeb14}.  However, due to the potential non-descent nature of the algorithmic scheme $\mathcal A$, it is unlikely to find such a sufficient decrease condition for $\mathcal A$ in general. Instead, based on the approximate descent property in ({\sf A}), a more elementary analysis can be applied to prove weak convergence results via properly invoking the diminishing step sizes. Such a strategy allows to avoid the typical sufficient decrease condition.  

Upon showing the weak convergence result, it is possible to infer several standard conclusions about the set of accumulation points   and to derive a uniformized version of the KL inequality. 

If $f$ further satisfies the KL inequality, then one can perform the last step:

\begin{enumerate}
	\item[({\sf C})] \textbf{Application of the KL inequality.} Based on the KL inequality of $f$, verify that the whole sequence $\{x^t\}_{t\geq1}$ converges to a single stationary point of $f$ and derive the corresponding convergence rates depending on the KL exponent $\theta$ and step size strategy. 
\end{enumerate} 

One core idea is to derive an auxiliary descent-type condition for $\{x^t\}_{t \geq 1}$ by accumulating the potential ascent (i.e., the non-descent terms $\kappa_3 \alpha_t^q$) in the approximate descent property of each step of the algorithm $\mathcal A$. Furthermore, the usage of the slightly stronger variant of the KL inequality in \cref{Def:KL-property}, \ref{B2}, or \ref{B3} allows to overcome the difficulties caused by the non-descent nature of $\{f(x^t)\}_{t\geq 1}$. 
By combining the standard KL analysis framework and the dynamics of the diminishing step sizes one can then derive strong limit-point convergence and the corresponding convergence rates depending on the KL exponent $\theta$ and chosen step sizes. 

\revise{\Cref{thm:finite sum} implies that} strong limit-point convergence is typically a consequence of the finite length property of $\{x^t\}_{t\geq1}$, i.e., $\kappa_2 \sum_t \|x^t - x^{t-1}\| <\infty$. However, the property ({\sf A}) allows the choice $\kappa_2 = 0$. In this case, one can still establish $\sum_t \alpha_t \|\nabla f(x^{t-1})\| <\infty$ \revise{as shown in the proof of \Cref{thm:finite sum}}. One possible way to prove strong limit-point convergence in this situation is to show a bound of the form $\|x^t -x^{t-1}\|\leq c_1 \alpha_t\|\nabla f(x^{t-1})\| + c_2 \alpha_t^{\bar p}$ with $\bar p \geq p$ and $c_1,c_2>0$. Such a bound can be mild for a large class of non-descent methods \cite{BerTsi00}. Strong convergence can then be obtained by further invoking the conditions on the diminishing step sizes. 

In principle, the above systematic strategy can be applied to non-descent methods that possess the approximate descent property ({\sf A}). One example is \RRp, as verified in \Cref{lemma:approx descent}. In the next subsection, we introduce another algorithm that also satisfies this approximate descent property, and hence it has the same strong limit-point convergence results shown in \Cref{sec:convergence analysis} as well.

\subsection{Reshuffled Proximal Point Method and its Convergence}\label{sec:prox point}

As an application of the systematic framework stated in the last subsection, we consider another non-descent method---namely, the reshuffled proximal point method ({\sf RPP})---and establish its strong limit-point convergence result under the KL inequality. 

{\sf RPP} has the same algorithmic framework as \RR (\Cref{algo:gm-shuff}) except for the updating step \eqref{eq:RS update}. {{\sf RPP}} replaces \eqref{eq:RS update} with the following proximal point step:
\be \label{eq:prox point update}
\tilde x_{i}^{t}\in \argmin_{x\in \R^n}  \ f(x, \sigma^{t}_{i}) + \frac{1}{2\alpha_{t}} \|x - \tilde x_{i-1}^t\|^2.
\ee
{{\sf RPP}} covers the incremental proximal point method (cf. \cite{bertsekas2011}) as a special case.  Let us first interpret the update \eqref{eq:prox point update} as a certain gradient-type step. The first-order optimality condition underlying $\tilde x^t_i$ implies that 
\be\label{eq:interpret PPM update}
\tilde x_{i}^{t} =  \tilde x_{i-1}^{t} - \alpha_{t} \nabla f(\tilde x_{i}^t, \sigma^{t}_{i}).
\ee
Clearly, this formula is similar to the update of \RRp. The only difference is that the gradient in \eqref{eq:interpret PPM update} is evaluated implicitly at $\tilde x_{i}^t$, while the gradient in \RR is evaluated at $\tilde x_{i-1}^t$. This observation motivates us to follow similar arguments to those of \Cref{lemma:descent} and \Cref{lemma:approx descent} for showing the approximate descent property of {{\sf RPP}}.
\begin{lemma}\label{lemma:approx descent PPM}
	Suppose that assumption \ref{A1} is valid.  Let $\{x^t\}_{t\geq 1}$ be the sequence generated by {{\sf RPP}} for solving problem \eqref{eq:problem} with step sizes $\{\alpha_t\}_{t\geq 1}$ satisfying
	$\alpha_t \in (0,{1}/{(\sqrt{2}{\sf L}N)}]$ and $\sum_{t=1}^\infty \alpha_t^3 < \infty$.
	Then, for all $t\geq 1$, we have
	\be\label{eq:approx descent PPM}
	f(x^t) \leq  f(x^{t-1})   - \frac{N\alpha_t}{2}\|\nabla{f}(x^{t-1})\|^2   -\frac{1 - \sL N \alpha_t}{2N \alpha_t}\left\| x^t-x^{t-1} \right\|^2 + 2\sG \sL^3 N^3  \alpha_t^3.
	\ee
	Here, $\sG := \left( f(x^{0})- \bar f_{\min} \right) \exp\left(\sum_{j=1}^{\infty}  6\sL^3 N^3 \alpha_j^3\right)$ is a finite positive constant.
\end{lemma}

As the proof of \Cref{lemma:approx descent PPM} is essentially identical to our derivation of \Cref{lemma:approx descent}, we will omit details here. \Cref{lemma:approx descent PPM} characterizes the approximate descent property of {{\sf RPP}}. Indeed, the recursion is the same as that of {\sf RR}. Therefore, we can follow our analysis framework ({\sf A})--({\sf C}) to show similar strong limit-point convergence results to that of \RR for {{\sf RPP}}, which we summarize in the following corollary. 

\begin{corollary}\label{coro:PPM} Suppose that the conditions \ref{A1}, \ref{C1}, and \ref{C2} are satisfied and let the stochastic process $\{X^t\}_{t \geq 1}$ be generated by {{\sf RPP}} utilizing step sizes of the form \eqref{eq:final-steps}. Then, $\{X^t\}_{t\geq 1}$ almost surely converges to a $\mathrm{crit}(f)$-valued random vector $X^* : \Omega \to \mathrm{crit}(f)$. Furthermore, let $\Theta : \Omega \to [0,1)$ denote the function that maps each sample $\omega \in \mathcal K$ to the KL exponent of $f$ at $X^*(\omega)$ and let $C : \Omega \to \R_+$ be the associated function of KL constants. Then, the event 
	\[ \mathcal H := \left\{ \omega \in \Omega: {\limsup}_{t \to \infty} \; t^{\Psi(\omega)} \cdot \|X^t(\omega) - X^*(\omega)\| < \infty \right\} \]
	occurs almost surely, where the rate function $\Psi : \Omega \to \R_+$ is given as in \eqref{eq:psi}.
\end{corollary}

\section{Conclusion} 
In this paper, we studied \RR  for smooth nonconvex optimization problems with a finite-sum structure. \RR is a practical algorithm and widely utilized.  Different from the existing convergence results for {\sf RR},  we established the first strong limit-point convergence results for \RR with appropriate diminishing step sizes under the KL inequality.  In addition, we derived the corresponding rate of convergence, depending on the KL exponent $\theta$ and on suitably selected diminishing step sizes $\{\alpha_t\}_{t\geq 1}$.  When $\theta \in [0,\frac12]$, \RR can achieve a rate of $\cO(t^{-1})$. When $\theta\in (\frac{1}{2}, 1)$, \RR has a convergence rate of the form $\cO(t^{-q})$ with $q\in (0,1)$.  Our results generalize the existing works that assume strong convexity to much broader nonconvex settings. 

Our results motivated a new KL analysis framework, which generalizes the standard one to cover a class of non-descent methods with diminishing step sizes. We summarized the main ingredients of our derivations in an informal analysis framework, which is of independent interest. This more general KL framework can potentially be applied to various other non-descent methods with diminishing step sizes.

\section*{Acknowledgments} \revise{We would like to thank the Associate Editor and three anonymous reviewers for their detailed and constructive comments, which have helped greatly to improve the quality and presentation of the manuscript.}

\appendix

\section{Proof of \cref{lemma:uniform KL}} \label{app:sec-unikl} 
\begin{proof} 
	The first steps of the proof are identical to the derivation of \cite[Lemma 6]{BolSabTeb14}. 
In particular, following the proof of \cite[Lemma 6]{BolSabTeb14} and using \cref{lemma:limit point set}, we can define a desingularizing function $\varrho \in \mathcal S_\eta$ of the form $\varrho(x) := \sum_{i=1}^m \varrho_i(x)$, $\varrho_i \in \mathcal Q_\eta$, $i = 1,\dots,m$ that satisfies the uniformized KL inequality \eqref{eq:uni-kl} for some $\delta, \eta, m > 0$.
%
	%
	%
	%
	%
	The proof is complete if we can show $\varrho\in\mathcal Q_\eta$. Let $C_{\rho_i} > 0$ denote the quasi-additivity constant in \eqref{eq:qrho} associated with each $\varrho_i$, $i=1,\dots,m$, and let $x,y \in (0,\eta)$ with $x+y < \eta$  be arbitrary. Without loss of generality, let us assume $x \leq y$.  Then, the concavity of $\varrho_i$ implies $\varrho_i^\prime(x) \geq \varrho_i^\prime(y)$ for all $i \in \{1,\dots,m\}$. In addition, let $j \in \{1,\dots,m\}$ be given with $\varrho_j^\prime(y) = \max_{1\leq i \leq m} \varrho_i^\prime(y)$. Using $\varrho_i^\prime(t) > 0$ for all $t \in (0,\eta)$, $i \in \{1,\dots,m\}$, and \eqref{eq:qrho} for $\varrho_j$, we obtain
	\begin{align*}  \frac{1}{\varrho^\prime(x+y)}  &= \frac{1}{\varrho^\prime_j(x+y) + \sum_{i \neq j}\varrho^\prime_i(x+y)} \leq \frac{1}{\varrho^\prime_j(x+y)}  \leq C_{\rho_j} \left[ \frac{1}{\varrho_j^\prime(x)} + \frac{1}{\varrho_j^\prime(y)} \right] \\&\leq \frac{2C_{\rho_j}}{\varrho^\prime_j(y)} \leq \frac{2m
			C_{\rho_j}}{\varrho^\prime(y)} \leq 2m \max_{1\leq i \leq m} C_{\rho_i} \left[ \frac{1}{\varrho^\prime(x)} + \frac{1}{\varrho^\prime(y)} \right]. \end{align*}
Thus, $\varrho$ satisfies condition \eqref{eq:qrho} with $C_\rho := 2m \max_{1\leq i \leq m} C_{\rho_i}$ and we have $\varrho \in \mathcal Q_\eta$.
\end{proof}

\section{Proof of \cref{lemma:step size}} \label{app:sec-pf-step}
\begin{proof}
	We first note that the improper integral $\int_{t}^\infty \frac{\alpha^3}{(y+\beta)^{3\gamma}} \, \mathrm{d}y$ is finite. By the integral comparison test, we have
	$
	\sum_{j=t}^\infty \alpha_j^3 \leq \alpha_t^3 + \int_{t}^\infty \frac{\alpha^3}{(y+\beta)^{3\gamma}} \, \mathrm{d}y = \alpha_t^3 + \frac{\alpha^3}{3\gamma-1}\frac{1}{(t+\beta)^{3\gamma-1}}
	$
	and $\sum_{j=t}^\infty \alpha_j^3 \geq  \int_{t}^\infty \frac{\alpha^3}{(y+\beta)^{3\gamma}} \, \mathrm{d}y = \frac{\alpha^3}{3\gamma-1}\frac{1}{(t+\beta)^{3\gamma-1}}$. Using the monotonicity and subadditivity of the mapping $x \mapsto x^\theta$ and the bounds on $\sum_{j=t}^\infty \alpha_j^3$, we can infer
	$
	\frac{a_\theta}{(t+\beta)^{(3\gamma-1)\theta}} \leq [{\sum}_{j=t}^\infty \alpha_j^3]^\theta \leq \alpha^{3\theta}_t + \frac{a_\theta}{(t+\beta)^{(3\gamma-1)\theta}},
	$
	where $a_\theta := \alpha^{3\theta}/(3\gamma-1)^\theta$.
	Thus, we obtain
    \[
	    \sum_{t=k}^\infty \frac{a_\theta \alpha}{(t+\beta)^{(3\gamma-1)\theta +\gamma}} \leq \sum_{t=k}^\infty \alpha_t \left[{\sum}_{j=t}^\infty \alpha_j^3\right]^\theta \leq  \sum_{t=k}^\infty \alpha_t^{1+3\theta} + \sum_{t=k}^\infty \frac{a_\theta\alpha}{(t+\beta)^{(3\gamma-1)\theta+\gamma}}. 
	\]
	Setting $\nu := (1+3\theta)\gamma - (1+\theta)$, we can utilize the integral comparison test again to derive 
	$
	\sum_{t=k}^\infty \alpha_t^{3\theta+1}  \leq \alpha_k^{1+3\theta} + \int_{k}^\infty \frac{\alpha^{1+3\theta}}{(y+\beta)^{(1+3\theta)\gamma}} \, \mathrm{d}y = \alpha_k^{1+3\theta} + \frac{k+\beta}{(1+3\theta)\gamma-1} \cdot \alpha_k^{1+3\theta}
	$ 
	and 
	$
	\frac{\nu^{-1}}{(k+\beta)^{\nu}} \leq \sum_{t=k}^\infty \frac{1}{(t+\beta)^{(3\gamma-1)\theta+\gamma}}  \leq \frac{1}{(k+\beta)^{(3\gamma-1)\theta+\gamma}} + \frac{\nu^{-1}}{(k+\beta)^{\nu}}
	$.
   Plugging these bounds into the estimate for $\sum_{t=k}^\infty \alpha_t [{\sum}_{j=t}^\infty \alpha_j^3]^\theta$, it follows
\begingroup
\allowdisplaybreaks
	\begin{align*} 
		& \hspace{-2ex} \frac{a_\theta\alpha \nu^{-1}}{(k+\beta)^{\nu}}  \leq {\sum}_{t=k}^\infty \alpha_t  \left[{\sum}_{j=t}^\infty \alpha_j^3\right]^\theta \\
		&\leq \left[ \alpha_k^{1+3\theta} + \frac{k+\beta}{(1+3\theta)\gamma-1} \cdot \alpha_k^{1+3\theta} \right] + \frac{a_\theta\alpha}{(k+\beta)^{(3\gamma-1)\theta+\gamma}} + \frac{a_\theta\alpha \nu^{-1}}{(k+\beta)^{\nu}} \\
		& = \left\{ a_\theta\left[\frac{\alpha}{k+\beta}+\frac{\alpha}{\nu}\right] +\left[\frac{\alpha^{1+3\theta}}{(k+\beta)^{1+\theta}} + \frac{\alpha^{1+3\theta}}{(\gamma+3\theta\gamma-1)(k+\beta)^{\theta}}\right]\right\} \cdot\frac{1}{(k+\beta)^{\nu}}\\
		&\leq \bar a_\theta\cdot \frac{1}{(k+\beta)^{\nu}}.
	\end{align*}
\endgroup
where $\bar a_\theta:= \frac{a_\theta\alpha}{1+\beta}+\frac{a_\theta\alpha}{\nu} +\frac{\alpha^{1+3\theta}}{(1+\beta)^{1+\theta}} + \frac{\alpha^{1+3\theta}}{(\gamma+3\theta\gamma-1)(1+\beta)^{\theta}}$.
	In summary, this yields
	\[
	\frac{\underline{a}_\theta}{(k+\beta)^{(1+3\theta)\gamma-(1+\theta)}} \leq {\sum}_{t=k}^\infty \alpha_t \left[{\sum}_{j=t}^\infty \alpha_j^3\right]^\theta \leq \frac{\bar{a}_\theta}{(k+\beta)^{(1+3\theta)\gamma-(1+\theta)}}, 
	\] 
	where $\underline{a}_\theta := a_\theta\alpha [(1+3\theta)\gamma-(1+\theta)]^{-1}$. 
\end{proof}

\revise{
\section{The Supermartingale Convergence Theorem} \label{app:sup-conv-thm}
We present the supermartingale convergence theorem \cite[Proposition A.31]{Ber16} below for convenience.
\begin{lemma}
	\label{lem:sup-conv-thm}
	Let $\{y_k\}_{k\geq 1},\{z_k\}_{k\geq 1},\{w_k\}_{k\geq 1}$, and $\{v_k\}_{k\geq 1}$ be sequences such that
	\[y_{k+1} \leq (1+v_k)y_k - z_k + w_k,\quad k=1,2,\dots,\]
	$\{z_k\}_{k\geq 1},\{w_k\}_{k\geq 1}$ and $\{v_k\}_{k\geq 1}$ are nonnegative, and
	\[{\sum}_{k=1}^\infty w_k<\infty,\quad {\sum}_{k=1}^\infty v_k<\infty.\]
	Then, either $y_k\to -\infty$, or else $\{y_k\}_{k\geq 1}$ converges to a finite value and ${\sum}_{k=1}^\infty z_k<\infty$.
\end{lemma}
}


\bibliographystyle{siamplain}
\bibliography{references}
\end{document}